\documentclass[11pt]{article}

\usepackage{amsmath,amsfonts,amssymb}
\usepackage{bm} 
\usepackage{graphicx}
\usepackage{url}
\usepackage{enumerate}

\usepackage[text={14.25cm,21.45cm},centering]{geometry}

\usepackage{amsthm}

\theoremstyle{plain}
\newtheorem{theorem}{Theorem}
\newtheorem{lemma}{Lemma}[section]
\newtheorem{cor}[lemma]{Corollary}

\theoremstyle{definition}
\newtheorem{defin}[lemma]{Definition}
\newtheorem{rem}[lemma]{Remark}
\newtheorem{exam1}[lemma]{Example}

\newtheoremstyle{example}
  {5pt}
  {5pt}
  {}
  {}
  {\bf}
  {.}
  {\newline}
  {}

\theoremstyle{example}
\newtheorem{exam}[lemma]{Example}

\newcommand{\Nset}{\mathbb{N}}
\newcommand{\Rset}{\mathbb{R}}
\newcommand{\Zset}{\mathbb{Z}}
\newcommand{\Qset}{\mathbb{Q}}
\newcommand{\Cset}{\mathbb{C}}

\newcommand{\PP}{\mathbb{P}}

\def\footnotesize{\small}

\long\def\symbolfootnote[#1]#2{\begingroup%
\def\thefootnote{\fnsymbol{footnote}}\footnote[#1]{#2}\endgroup}

\title{Finite-state Markov Chains Obey Benford's Law}
\author{%
Bahar Kaynar\\
{\it Vrije Universiteit, Amsterdam, The Netherlands}\\
\and
Arno Berger\\
{\it University of Alberta, Edmonton, Canada}\\
\and
Theodore P. Hill\\
{\it Georgia Institute of Technology, Atlanta, USA}\\
\and
Ad Ridder\\
{\it Vrije Universiteit, Amsterdam, The Netherlands}\\
}
\date{\today}

\begin{document}
\maketitle
%

\begin{abstract}
A sequence of real numbers $(x_n)$ is Benford if the significands,
i.e.\ the fraction parts in the floating-point representation of
$(x_n)$, are distributed logarithmically. Similarly, a discrete-time
irreducible and aperiodic finite-state Markov chain with probability
transition matrix $P$ and limiting matrix $P^*$ is Benford if every
component of both sequences of matrices $(P^n - P^*)$ and
$(P^{n+1}-P^n)$ is Benford or eventually zero. Using recent tools
that established Benford behavior both for Newton's method and for
finite-dimensional linear maps, via the classical theories of uniform
distribution modulo $1$ and Perron-Frobenius, this paper derives a
simple sufficient condition (``nonresonance'') guaranteeing that $P$,
or the Markov chain associated with it, is Benford. This result in
turn is used to show that almost all Markov chains are Benford, in
the sense that if the transition probabilities are chosen independently
and continuously, then the resulting Markov chain is Benford with
probability one. Concrete examples illustrate the various cases that
arise, and the theory is complemented with several simulations and
potential applications.
\end{abstract}


\symbolfootnote[0]{ \noindent Kaynar and Hill were supported
by grants 400-06-044 and 040-11-091 from The Netherlands Organization
for Scientific Research (NWO); Berger was supported by an NSERC Discovery
Grant.
\par Corresponding author's address: Ad Ridder, Department of
Econometrics and Operations Research, Vrije Universiteit Amsterdam,
The Netherlands; email: \url{aridder@feweb.vu.nl}}

\small
\renewcommand{\baselinestretch}{1.4}
\renewcommand{\arraystretch}{0.75}
\normalsize
%

\section{Introduction}\label{s:intro}

Benford's Law (BL) is the widely-known logarithmic probability
distribution on significant digits (or equivalently, on significands),
and its most familiar form is the special case of first significant
digits (base $10$), namely,
\begin{equation}
\label{eq:1} \PP \left({D_1=d_1}\right)
= \log_{10}\left( 1 + \frac{1}{d_1}\right),
\quad \forall d_1\in \{ 1,2,\ldots,9\} \, ,
\end{equation}
where for each $x\in \Rset^+$, the number $D_1(x)$ is the \emph{first
significant digit\/} (base $10$) of $x$, i.e.\ the unique integer $d
\in \{1,2,\ldots,9 \}$ satisfying $10^k d \le x < 10^k (d+1)$ for
some, necessarily unique, $k \in \Zset$. Thus, for example,
$D_1(30122)=D_1(0.030122)=D_1(3.0122)=3$, and \eqref{eq:1} implies
that
$$
\PP(D_1=1)  = \log_{10}2 \cong 0.301 \, , \quad
\PP(D_1=2)  = \log_{10}(3/2) \cong 0.176\, , \quad \mbox{etc.,}
$$
see also Table \ref{tab:5} below.

In a form more complete than \eqref{eq:1}, BL is a statement about
joint distributions of the first $n$ significant digits (base $10$)
for any $n\in\Nset$, namely,
\begin{align}
\PP \bigl( (D_1,D_2,D_3,& \ldots,D_n)  =  (d_1,d_2,d_3,\ldots,d_n)\bigr) \nonumber \\
     & =\log_{10} \left( \sum\nolimits_{j=1}^n 10^{n-j}d_j + 1\right)-\log_{10}
     \left( \sum\nolimits_{j=1}^n 10^{n-j}d_j \right) \label{eq:2} \\
     & = \log_{10} \left( 1+ \frac{1}{\sum_{j=1}^n 10^{n-j}d_j}\right), \nonumber
\end{align}
where $d_1\in \{1,2,\ldots,9\}$ and $d_j \in \{0,1,2,\ldots,9\}$ for
$j\ge 2$, and $D_2,D_3$, etc.\ represent the second, third, etc.\
\emph{significant digit functions} (base $10$). Thus, for example,
$D_2(30122)=D_2(0.030122)=D_2(3.0122)=0$, and a special case of
\eqref{eq:2} is
\[
\PP\bigl((D_1,D_2,D_3)=(3,0,1)\bigr) =
\log_{10}302-\log_{10}301=\log_{10}\left( 1 +
\frac{1}{301}\right) \cong 0.00144\, .
\]
Formally, for every $n \in \Nset$, $n \ge 2$, the number $D_n(x)$, the
\emph{n-th significant digit\/} (base $10$) of $x\in\Rset^+$, is defined
inductively as the unique integer $d \in \{0,1,2,\ldots,9\}$ such that
\[
10^k \left( d+  \sum\nolimits_{j=1}^{n-1} 10^{n-j} D_j(x) \right) \le x <
10^k \left( d+1 + \sum\nolimits_{j=1}^{n-1} 10^{n-j} D_j(x) \right)
\]
for some (unique) $k \in \Zset$.

The formal probability framework for the significant-digit law is
described in \cite{Hill95,Hill96}. The sample space is the set of
positive reals, and the $\sigma$-algebra of events is the
$\sigma$-algebra generated by the (decimal) \emph{significand} (or
\emph{mantissa}) \emph{function} $S : \Rset^+ \rightarrow [1,10)$,
where $S(x)$ is the unique number $s \in [1,10)$ such that $x=10^k
s$ for some $k \in \Zset$. Equivalently, the significand events are
the sets in the $\sigma$-algebra generated by the significant digit
functions $D_1,D_2, D_3$, etc. The probability measure on this sample
space associated with BL is
$$
\PP \left(S \le t \right) = \log_{10} t\, , \quad \forall
t \in [1,10) \, .
$$
It is easy to see that the significant digit functions $D_1$ and
$D_2,D_3$, etc.\ are well-defined $\{1,2, \dots, 9\}$- and
$\{0,1,2,\ldots,9 \}$-valued random variables, respectively, on this
probability space with probability mass functions as given in
\eqref{eq:1} and \eqref{eq:2}.
\par\medskip\noindent
\textbf{Note.} Throughout this article, all results are restricted to
decimal (base $10$) significant digits, and accordingly $\log$ always
denotes the base $10$ logarithm. For notational convenience, $D_n(0):=0$
for all $n\in \Nset$. The results carry over easily to arbitrary bases
$b \in \Nset\setminus \{ 1\}$, as is evident from \cite{Berger05b},
where the essential difference is replacing $\log_{10}$ by $\log_b$,
and the decimal significant digits by the base-$b$ significant digits.
\par\medskip\noindent
Benford's Law is now known to hold in great generality, e.g.\ for
classical combinatorial sequences such as $(2^n)$, $(n!)$ and the
Fibonacci numbers $(F_n)$; iterations of linearly- or
nonlinearly-dominated functions; solutions of ordinary differential
equations; products of independent random variables; random mixtures
of data; and random maps (e.g., see
\cite{Berger05a,Berger07,Berger08,Diaconis79,Hill96}).
Table \ref{tab:5} compares the empirical frequencies of $D_1$ for the
first $1000$ terms of the sequences $(2^n)$, $(n!)$ and $(F_n)$. These
empirical frequencies illustrate what it means to follow BL and also
foreshadow the simulations in Section \ref{s:sim}.

\begin{table}[h]
\vspace*{4mm}
\begin{center}
\small{
\begin{tabular}{c c c c | c }
  \cline{1-5}
  & & & & \\[-3mm]
  $D_1$ & $(n!)$ & $(2^n)$
  & $(F_n)$ &  Benford \\[2mm]
  \hline \hline
  & & & & \\[-3mm]
  1 & 0.293 & 0.292 & 0.301 &  0.30103 \\
  2 & 0.176 & 0.180 & 0.176 &  0.17609 \\
  3 & 0.124 & 0.126 & 0.126 &  0.12493 \\
  4 & 0.102 & 0.098 & 0.096 &  0.09691 \\
  5 & 0.087 & 0.081 & 0.079 &  0.07918 \\
  6 & 0.069 & 0.068 & 0.067 &  0.06694 \\
  7 & 0.051 & 0.057 & 0.057 &  0.05799 \\
  8 & 0.051 & 0.053 & 0.053 &  0.05115 \\
  9 & 0.047 & 0.045 & 0.045 &  0.04575 \\[1.5mm]
  \hline \hline\\[-6mm]
\end{tabular}
}
\end{center}
\caption{\textit{Empirical frequencies of $D_1$ for the first $1000$
terms of the sequences $(2^n)$, $(n!)$ and the Fibonacci numbers
$(F_n)$, as compared with the Benford probabilities.}} \label{tab:5}
\end{table}
The main contribution of this article is to adapt recent results on
BL in the multi-dimensional setting (\cite{Berger05b}) in order to
establish BL in finite-dimensional, time-homogeneous Markov chains,
and to suggest several applications including error analysis in
numerical simulations of $n$-step transition matrices.
\par
Concretely, given the transition matrix $P$ of a finite-state Markov
chain (i.e., $P$ is a row-stochastic matrix), a common problem is to
estimate the limit $P^* = \lim_{n \rightarrow \infty} P^n$. The two
main theoretical results below, Theorems \ref{thm:1} and \ref{thm:2},
respectively, show that under a natural condition (``nonresonance'')
every component of the sequence of matrices $(P^n-P^*)$ and
$(P^{n+1}-P^{n})$ obeys BL, and that this behavior is typical, i.e.,
it occurs for almost all Markov chains. Simulations are provided for
illustration, followed by several potential applications including
the estimation of roundoff errors incurred when estimating $P^*$ from
$P^n$, and possible (partial negative) statistical tests to decide
whether data comes from a finite-state Markov process.

%

\section{Benford Markov chains and main tools}\label{s:main}

The set of natural, integer, rational, positive real, real and
complex numbers are symbolized by $\Nset, \Zset, \Qset, \Rset^+,
\Rset$ and $\Cset$, respectively. The real part, imaginary part,
complex conjugate and absolute value (modulus) of a number $z\in
\Cset$ is denoted by $\mathfrak{Re}z,\mathfrak{Im}z,\bar{z}$ and
$|z|$, respectively. For $z \neq 0$, the argument $\arg z$ is the
unique number in $(-\pi,\pi ]$ that satisfies $z = |z|e^{i \arg
z}$. For ease of notation, $\arg 0:= 0$ and $\log 0 := 0$. The
cardinality of the finite set $A$ is $\# A$. Throughout this article,
the sequence $\bigl(a(1),a(2), a(3), \ldots\bigr)$ is denoted by
$\bigl(a(n)\bigr)$. Thus, for example,
$(\alpha^n)=(\alpha^1,\alpha^2,\alpha^3,\ldots)$ and $\left( P^{n+1} -
P^n\right)= \left(P^2-P^1, P^3-P^2, P^4-P^3,\ldots\right)$. Boldface
symbols indicate randomized quantities, e.g.\ $\bm{X}$ denotes a random
variable or vector and $\bm{P}$ a random transition probability matrix.
\begin{defin}
\label{def:1} A sequence $(x_n)$ of real numbers is \emph{Benford}
(``\emph{follows BL}'') if
\[
\lim\nolimits_{n \rightarrow \infty} \frac{\# \{ j \le n : S(|x_j|)\le t
\}}{n} = \log t\, , \quad \forall t \in [1,10)\, .
\]
\end{defin}
\par\medskip

\noindent
The main subject of this paper is the Benford behavior of finite-state
Markov chains. The theory uses three main tools: the classical theory
of \emph{uniform distribution modulo 1}, see e.g.\ \cite{Kuipers74};
recent results for BL in one- and multi-dimensional dynamical systems
(\cite{Berger05a,Berger05b}); and the classical Perron-Frobenius theory
for Markov chains, see e.g.\ \cite{Bremaud99,Serre02}. The first lemma
records the relationship between uniform distribution theory and BL,
and the second lemma is an application establishing BL for certain
basic sequences that will be used repeatedly below. Here and throughout,
the term \emph{uniformly distributed modulo 1} is abbreviated as
\emph{u.d.} mod 1.

\begin{lemma}[\cite{Diaconis79}]
\label{lem:3}  A sequence $(x_n)$ of real numbers is Benford if and
only if $\left(\log |x_n|\right)$ is u.d.\ {\rm mod} $1$.
\end{lemma}

\noindent An immediate application of Lemma~\ref{lem:3} is the following
useful lemma.

\begin{lemma}[\cite{Berger05a}]
\label{lem:4} Let $(x_n)$ be Benford. Then for all $\alpha \in \Rset$
and $k \in \Zset$ with $\alpha k \neq 0$, the sequence $(\alpha x_n^k)$
is also Benford.
\end{lemma}

Lemmas \ref{lem:3} and \ref{lem:4} are fundamental tools for analyzing
BL in the setting of multi-dimensional dynamical systems (\cite{Berger05b}),
and although those results do not apply directly to the Markov chain
setting, the first part of the theory established below relies heavily
on those ideas specialized to the case of row-stochastic matrices.
\par
The next lemma follows easily from known results. It is included
here since these observations play a central role in determining
whether a Markov chain is Benford, as illustrated in the three
examples following the lemma. Stronger conclusions are possible,
as suggested in Example~\ref{ex:14}(iii) below, but are not needed
here.

\begin{lemma}
\label{lem:5} Let $a, b, \alpha, \beta$ be real numbers with $a\ne 0$
and $|\alpha| > |\beta|$. Then $(a \alpha^n + b \beta^n)$ is Benford
if and only if $\log|\alpha|$ is irrational.
\end{lemma}

\begin{proof}
Since $|\alpha| > |\beta|$, the significands of $\alpha^n$ dominate
those of $\beta^n$ asymptotically, so the conclusion follows from
Lemma \ref{lem:3}, Lemma \ref{lem:4} and Weyl's classical theorem
that iterations of an irrational rotation on the circle are
uniformly distributed.
\end{proof}

\begin{exam1}\label{ex:14}
\mbox{}
\begin{enumerate}[(i)]
\item
The sequences $(2^n)$, $(0.2^n)$, $(3^n)$, $(0.3^n)$ are Benford, whereas
$(10^n)$, $(0.1^n)$, $\left(\sqrt{10}^n\right)$ are not Benford.
\item The sequence $\bigl(0.01 \cdot 0.2^n + 0.2 \cdot 0.01^n\bigr)$ is
Benford, whereas $\bigl( 0.1 \cdot 0.02^n + 0.02 \cdot 0.1^n\bigr)$ is
not Benford.
\item
The sequence $\bigl( 0.2^n + (-0.2)^n\bigr)$ is not Benford, since all
odd terms are zero, but $\bigl( 0.2^n + (-0.2)^n + 0.03^n \bigr)$ is Benford
--- although this does not follow directly from Lemma \ref{lem:5}.
\end{enumerate}
\end{exam1}

\noindent
\textbf{Notation.} For every integer $d > 1$, the set of all row-stochastic
matrices of size $d\times d$ is denoted by $\mathcal{P}_d$.

\medskip
Now, let $P \in \mathcal{P}_d$ be the transition probability matrix of a
Markov chain. All Markov chains (or their associated matrices $P$)
considered in this work are assumed to be finite-state (with $d>1$
states), irreducible and aperiodic. Let $\lambda_1,\ldots,\lambda_s$,
$s \leq d$, be the \emph{distinct\/} (possibly non-real) eigenvalues
of the stochastic matrix $P$, with corresponding spectrum
$\sigma(P)=\{\lambda_1,\ldots,\lambda_s \}$, i.e., $\sigma(P)$ is
the set of all distinct eigenvalues. Accordingly, the set
$\sigma(P)^+ = \{ \lambda \in \sigma(P) : \mathfrak{Im}\lambda \ge 0 \}$
forms the ``upper half'' of the spectrum. The usage of $\sigma(P)^+$
refers to the fact that non-real eigenvalues of real matrices always
occur in conjugate pairs, so the set $\sigma(P)^+$ only includes one
of the conjugates. Without loss of generality, throughout this work it
is also assumed that the eigenvalues in $\sigma(P)$ are labeled such that
\[
|\lambda_1| \ge |\lambda_2| \ge \ldots \ge |\lambda_s|\, .
\]
Furthermore, the column vectors $u_1,\ldots,u_s$ and $v_1,\ldots,v_s$
denote associated sequences of left and right eigenvectors, respectively.
The third main tool in this paper is the classical Perron-Frobenius
theory of Markov chains, and the following lemma summarizes some of
the special properties of transition matrices for ease of reference.
\begin{lemma}
\label{lem:6} Suppose $P \in \mathcal{P}_d$ is irreducible and aperiodic.
Then $\lambda_1 = 1 > |\lambda_{\ell}|$ for all $\ell=2,\ldots,s$, and
there exists a $P^*\in \mathcal{P}_d$ such that
\begin{enumerate}
\item $\lim_{n \rightarrow \infty} P^n = P^*$;
\item for every $n\in \Nset$,
\begin{equation}
\label{eq:7}P^n - P^*  = \lambda_2^n C_2 + \ldots + \lambda_s^n C_s\, ,
\end{equation}
where each $C_\ell$ is a $d \times d$-matrix whose components $C_{\ell}^{(i,j)}$
are polynomials in $n$ with complex coefficients and degrees $k_\ell^{(i,j)} < d$.
\end{enumerate}
\end{lemma}

\begin{proof}
Immediate from the Perron-Frobenius theorem, see e.g.\ $\cite{Seneta81}$.
\end{proof}

\noindent
The second dominant eigenvalue $\lambda_2$ plays an important role
whenever $C_2^{(i,j)} \neq 0$. The analysis is especially straightforward
if all eigenvalues are simple, i.e., if $\# \sigma(P) = d$. In this case,
for every $n\in \Nset$,
\begin{equation}\label{eq:specdecom}
P^n - P^* =
\sum\nolimits_{\ell = 2}^d \lambda_\ell^n B_{\ell}
\quad \mbox{ and } \quad
P^{n+1}-P^n
= \sum\nolimits_{\ell=2}^{d} \lambda_{\ell}^n (\lambda_{\ell}-1)B_{\ell}
\end{equation}
holds with the $d-1$ matrices $B_{\ell}= v_{\ell} u_{\ell}^\top/ v_{\ell}^\top u_{\ell}
\in \Cset^{d\times d}$. Next is the key definition in this paper.
\begin{defin}
\label{def:2} A Markov chain, or its associated transition probability
matrix $P$, is \emph{Benford} if each component of $\left(P^n -
P^* \right)$ and $\left(P^{n+1} - P^{n} \right)$ is either Benford or
eventually zero.
\end{defin}

\noindent
The following examples illustrate the notions of Benford and non-Benford
Markov chains.
\begin{exam1} \label{ex:2} (Examples of Benford Markov chains)
\begin{enumerate}
\item Let $d=2$ and $ P= \left[ \begin{array}{cc}
           0.7  &  0.3 \\
           0.4  &  0.6
           \end{array}\right]
       $.
By \cite[p.\ 432]{Feller50}, $ P^* = {\displaystyle \frac17} \left[ \begin{array}{cc}
           4  &  3 \\
           4  &  3
           \end{array}\right]$, and
\[
P^n - P^* = \frac{0.3^n}{7} \left[ \begin{array}{rr}
                 3  &  -3 \\
                -4  &  4
                \end{array} \right] \quad \mbox{and} \quad
P^{n+1} - P^{n} =  0.3^{n} \left[ \begin{array}{rr}
                 -0.3  &  0.3 \\
                0.4  &  -0.4
                \end{array} \right]
\]
holds for all $n\in \Nset$.
In both sequences every component is a multiple of $(0.3^n)$, and hence
Benford by Lemma \ref{lem:5} since $\log 0.3$ is irrational. The
two-dimensional case will be discussed in more generality in Examples
\ref{ex:7} and \ref{ex:9}.
\item Let $d=3$ and $ P= \left[ \begin{array}{ccc}
           0.9  &  0.0 & 0.1 \\
           0.6  &  0.3 & 0.1 \\
           0.1  &  0.0 & 0.9
\end{array}\right]
$. It is easy to check via spectral decomposition (e.g.\
\cite{Bremaud99}) that the eigenvalues of $P$ are $\lambda_1
= 1$, $\lambda_2 = 0.8$ and $\lambda_3 = 0.3$, and
$P^* = \left[ \begin{array}{ccc}
              0.5   &              0   &  0.5 \\
              0.5   &              0   &  0.5 \\
              0.5   &              0   &  0.5
\end{array}\right]$. The three eigenvalues are distinct, leading to
$$
P^n - P^*  =  0.8^n
\left[ \begin{array}{rrr}
              0.5   &               0  & -0.5 \\
              0.5   &              0   & -0.5 \\
              -0.5   &              0   &  0.5
\end{array}\right]  + 0.3^n
\left[ \begin{array}{rrr}
               0 &  0 & 0 \\
               -1 & 1 & 0 \\
               0 & 0 & 0
\end{array}\right] \, ,
$$
as well as
$$
P^{n+1} - P^{n}  = 0.8^{n}
\left[ \begin{array}{rrr}
-0.1   &              0   &  0.1 \\
-0.1   &              0   &  0.1 \\
 0.1   &              0   & -0.1
\end{array}\right]  +  0.3^{n}
\left[ \begin{array}{rrr}
0 &  0 & 0 \\
0.7 & -0.7 & 0 \\
0 & 0 & 0
\end{array}\right] \, .
$$
As can be seen directly, in both cases the components $(1,2)$
and $(3,2)$ are zero for all $n$, whereas by Lemma \ref{lem:5}
all other components follow BL. Hence, the Markov chain defined by the
transition probability matrix $P$ is Benford.
\par
As will be observed later, the moduli of the eigenvalues as well as
a specific rational relationship between them play a crucial
role in the analysis of BL in Markov chains, similar to the results
in \cite{Berger05b}.
\end{enumerate}
\end{exam1}
\begin{exam1}
\label{ex:3} (Examples of non-Benford Markov chains)
\begin{enumerate}[(i)]
\item Let $d=2$ and
$ P= \left[\begin{array}{cc}
           0.2  &  0.8 \\
           0.1  &  0.9
\end{array}\right]
$, hence $P^*= {\displaystyle \frac19} \left[\begin{array}{cc}
           1  &  8 \\
           1  &  8
\end{array}\right]$ and, for every $n\in \Nset$,
\[
P^n - P^* = \frac{0.1^n}{9}
\left[\begin{array}{rr}
           8  &  -8 \\
           -1  &  1
\end{array}\right]
\quad \mbox{and} \quad
P^{n+1} - P^{n} = 0.1^{n}
\left[ \begin{array}{rr}
           -0.8  &  0.8 \\
           0.1  &  -0.1
\end{array}\right] \, .
\]
Since $\log 0.1$ is rational, Lemma \ref{lem:5} implies that no component
of $\left( P^n - P^*\right)$ or $ \left( P^{n+1} - P^{n} \right)$ is
Benford. For example, $D_1\bigl( |( P^{n}-P^*)^{(1,1)}|\bigr) = 8$
for all $n\in \Nset$.

\item Let $d=3$ and $
P= \left[ \begin{array}{ccc}
           0.0  &  0.1 & 0.9 \\
           0.1  &  0.3 & 0.6 \\
           0.1  &  0.1 & 0.8
\end{array}\right]
$. The eigenvalues of $P$ are $\lambda_1=1$, $\lambda_2=0.2$ and
$\lambda_3 = -0.1$. Since these three eigenvalues are distinct, again
by spectral decomposition,
$$
P^n - P^* = \frac{0.2^n}{8}
\left[ \begin{array}{rrr}
           0   &  -1  &  1 \\
           0   &   7  & -7 \\
           0   &  -1  &  1
\end{array}\right]+ \frac{(-0.1)^n}{11}
\left[\begin{array}{rrr}
            10 &  0 & -10 \\
            -1 &  0 &  1  \\
            -1 &  0 & 1
\end{array} \right] \, ,
$$
as well as
$$
P^{n+1} - P^{n}  = 0.2^{n}
\left[ \begin{array}{rrr}
          0    &  0.1  & -0.1 \\
          0   &  -0.7  &  0.7 \\
          0   &   0.1  & -0.1
\end{array}\right]  +  (-0.1)^{n}
\left[ \begin{array}{rrr}
         -1  &  0 &   1 \\
         0.1 &  0 & -0.1\\
         0.1 &  0 & -0.1
\end{array} \right] \, .
$$
The first column of $B_2$ is zero, hence for that column the
relevant eigenvalue is $\lambda_3 = -0.1$. Since $\log 0.1$ is
rational, no component in the first column of either sequence
$\left( P^n - P^* \right)$ and $\left( P^{n+1} - P^{n} \right)$
follows BL, i.e., $P$ is not Benford.
\end{enumerate}
\end{exam1}


\section{Sufficient condition that a Markov chain is Benford}\label{s:thm1}

To analyze the behavior of the sequences $\left( P^n - P^* \right)$ and
$ \left( P^{n+1} - P^{n} \right)$ associated with a Markov chain, a
nonresonance condition on $P$ will be helpful. Recall that real numbers
$x_1,\ldots,x_k$ are \emph{rationally independent} (or $\Qset$-\emph{independent})
if $\sum_{j=1}^k q_j x_j=0$ with $q_1,\ldots,q_k\in\Qset$ implies that
$q_j=0$ for all $j=1, \ldots, k$; otherwise $x_1, \ldots, x_k$ are {\em rationally
dependent}.
\begin{defin}\label{def:4}
A stochastic matrix $P$ is \emph{nonresonant} if every nonempty subset
$\Lambda_0=\{ \lambda_{i_1}, \ldots , \lambda_{i_k} \} \subset \sigma(P)^+
\setminus \{\lambda_1\}$ with $|\lambda_{i_1}|=\ldots=|\lambda_{i_k}| = L_0$
satisfies $\# (\Lambda_0 \cap \Rset) \le 1$, and the numbers $1$, $\log L_0$
and the elements of $\frac1{2\pi} \arg \Lambda_0$ are rationally independent,
where
\[
{\textstyle \frac{1}{2\pi}}\arg \Lambda_0 := \left\{ {\textstyle\frac{1}{2\pi}} \arg
\lambda_{i_1},\ldots,{\textstyle\frac{1}{2\pi}}\arg \lambda_{i_k} \right\} \setminus
\left\{0,{\textstyle\frac{1}{2}} \right\}\, .
\]
A Markov chain is nonresonant whenever its transition probability matrix
is. A stochastic matrix or Markov chain is {\em resonant\/} if
it is not nonresonant.
\end{defin}

\noindent
Notice that for $P$ to be nonresonant, it is required specifically that the
logarithms of the moduli of all the eigenvalues other than $\lambda_1 = 1$
are irrational; in particular, $P$ has to be invertible. Theorem \ref{thm:1}
below establishes that nonresonance is sufficient for $P$ to be Benford.
There is a close correspondence between Definition \ref{def:4} of a nonresonant
matrix and the notion of a matrix not having $10$-{\em resonant spectrum},
as introduced in \cite{Berger05b}. The main difference is that the eigenvalue
$\lambda_1 = 1$ is excluded in Definition \ref{def:4}, whereas every stochastic
matrix has $10$-resonant spectrum.
\begin{exam1}
\label{ex:5} (Examples of nonresonant matrices)
\begin{enumerate}[(i)]
\item
Both transition matrices in Example~\ref{ex:2} are
nonresonant.
\item
Let $d=5$ and $ P= \left[ \begin{array}{lllll}
                    0.0  &  0.25 & 0.25 & 0.25 & 0.25 \\
                    0.25 &  0.0  & 0.25 & 0.25 & 0.25 \\
                    0.25 &  0.25 & 0.0  & 0.25 & 0.25 \\
                    0.25 &  0.25 & 0.25 & 0.0  & 0.25 \\
                    0.25 &  0.25 & 0.25 & 0.25 & 0.0
\end{array}\right]
$. The eigenvalues of $P$ are $\lambda_1=1$ and
$\lambda_2= -0.25$ (with multiplicity four), so  $\Lambda_0
= \{ -0.25 \}$, with $L_0=0.25$ and $\frac{1}{2\pi}\arg
\Lambda_0=\emptyset$. Since $\log 0.25$ is irrational, $P$ is
nonresonant.
\end{enumerate}
\end{exam1}

\begin{exam1}
\label{ex:6} (Examples of resonant matrices)
\begin{enumerate}[(i)]
\item Two real eigenvalues of opposite sign:
Let $d=3$ and $ P=\left[ \begin{array}{ccc}
    0.6 &   0.4 &   0.0  \\
    0.8 &   0.0 &   0.2\\
    0.0 &   0.6 &   0.4
\end{array}\right]
$.
The eigenvalues of $P$ are $\lambda_1=1$ and $\lambda_{2,3}= \pm\sqrt{0.2}$.
Notice that $\log|\lambda_2|=\log|\lambda_3|=-\frac12\log 5$ is
irrational. With $\Lambda_0 = \{ \sqrt{0.2}, - \sqrt{0.2}\}$ clearly
$\# (\Lambda_0 \cap \Rset)=2$, hence $P$ is resonant. The spectral decomposition
\eqref{eq:specdecom} yields
\[ (P^n-P^*)^{(1,1)} = 0.2\lambda_2^n + 0.2\lambda_3^n=
\begin{cases}
0.4\left(\sqrt{0.2}\right)^n  & \mbox{ if } n \mbox{ is even},\\
0  & \mbox{ if } n \mbox{ is odd},
\end{cases}
\]
showing that $P$ is not Benford either.
\item
Eigenvalues with rational logarithms:
Let $d=3$ and
$
P= \left[\begin{array}{ccc}
        0.0 &  0.1 & 0.9 \\
        0.5 &  0.1 & 0.4 \\
        0.3 &  0.3 & 0.4
\end{array}\right]
$. The eigenvalues are $\lambda_1 = 1$ and $ \lambda_{2,3} =
-0.25 \pm 0.05i\sqrt{15}$. Since $\log |\lambda_{2,3}|= -0.5$ is rational,
the matrix $P$ is resonant.
\item
Eigenvalues with rational argument: Let $d=3$ and
$
P = \left[ \begin{array}{ccc}
0.3 & 0.3 & 0.4\\
0.3 & 0.5 & 0.2 \\
0.1 & 0.7 & 0.2
\end{array}\right]
$. The eigenvalues are $\lambda_1 = 1$ and $\lambda_{2,3}= \pm 0.2i$. Note that
$\log|0.2i|= -1+\log 2$ is irrational, but $\frac1{2\pi} \arg(0.2i)=\tfrac14$ is
rational. Thus $P$ is resonant. Spectral decomposition gives
$B_1^{(2,2)}=B_2^{(2,2)}=\tfrac14$, hence
\[
(P^n-P^*)^{(2,2)} =\tfrac14 \bigl( (0.2i)^n+(-0.2i)^n\bigr)
=\begin{cases}
\tfrac12 \cdot (-1)^{n/2} \cdot 0.2^n & \quad\text{if } n \mbox{ is even}, \\
0 & \quad\text{if } n \mbox{ is odd},
\end{cases}
\]
which in turn shows that $P$ is not Benford.

\item
Eigenvalues leading to rational dependencies within
$\{1, \log L_0\}\cup \frac1{2\pi}\arg \Lambda_0$: Let $d=7$ and
$ P=\left[ \begin{array}{ccccccc}
    0.2  &  0.1 &   0.0 &   0.0 &   0.1 &   0.0 &   0.6\\
    0.1  &  0.1 &   0.1 &   0.1 &   0.2 &   0.0 &   0.4\\
    0.1  &  0.1 &   0.1 &   0.1 &   0.1 &   0.2 &   0.3\\
    0.0  &  0.2 &   0.3 &   0.0 &   0.2 &   0.0 &   0.3\\
    0.1  &  0.2 &   0.1 &   0.1 &   0.0 &   0.1 &   0.4\\
    0.2  &  0.0 &   0.2 &   0.1 &   0.1 &   0.0 &   0.4\\
    0.1  &  0.2 &   0.2 &   0.0 &   0.0 &   0.0 &   0.5
\end{array}\right]
$. The characteristic polynomial $\psi_P$ of $P$ factors as
$$
\psi_P(\lambda) = (\lambda - 1)\,\left(\lambda^2+0.1\lambda-0.01\right)\,
\left(\lambda^2-0.01(2-i)\right)\, \left(\lambda^2-0.01(2+i)\right) \, .
$$
The roots of the second factor are
$-\frac{1}{20}\left(1 \pm \sqrt{5}\right)$; the third factor has roots
\[ \pm {\textstyle \frac{1}{10}}\sqrt{2-i}=
\pm {\textstyle \frac{1}{20}}\left(\sqrt{4+2\sqrt{5}} - i\sqrt{-4+2\sqrt{5}}
\right)\, ,
\]
and the fourth factor has roots
\[ {\textstyle \pm \frac{1}{10}} \sqrt{2+i}=
\pm {\textstyle \frac{1}{20}} \left(\sqrt{4+2\sqrt{5}} + i\sqrt{-4+2\sqrt{5}}
\right)\, .
\]
Thus, the dominated positive spectrum is
\[   \sigma(P)^+\setminus\{\lambda_1\} =
 {\textstyle \frac{1}{20}} \left\{
 -(\sqrt{5}+1),\, \sqrt{5}-1,\,
 -2\sqrt{2-i},\, 2\sqrt{2+i}
 \right\}.\]
Clearly, the logarithms of the absolute values of the two real
eigenvalues are irrational. The four non-real eigenvalues all have the
same modulus $L_0 = \frac{1}{10}5^{1/4}$ (different from the two
real eigenvalues), and $\log L_0 = -1+\frac14\log 5$ is
irrational. Let $\Lambda_0=\frac{1}{10} \left\{-\sqrt{2-i},\, \sqrt{2+i}\right\}$.
Notice that $\arg(2\mp i)=\mp \arctan\frac12$, so
\[
{\textstyle \frac1{2\pi}} \arg \Lambda_0 = \left\{ {\textstyle \frac12} -
{\textstyle \frac1{4\pi}} \arctan {\textstyle \frac12} ,\;{\textstyle \frac1{4\pi}}
\arctan {\textstyle \frac12}
\right\} =: \{x_3, x_4 \}\, .
\]
Since
$$
-1 \cdot 1 + 0 \cdot \log L_0 + 2 \cdot x_3 + 2 \cdot x_4 = 0 \, ,
$$
the elements of $\{1, \log L_0\}\cup \frac1{2\pi}\Lambda_0$ are $\Qset$-dependent,
and hence $P$ is resonant.
\end{enumerate}
\end{exam1}

\par\medskip
\noindent
The first main theoretical result of this paper is
\begin{theorem}
\label{thm:1} Every nonresonant
irreducible and aperiodic finite-state Markov chain is Benford.
\end{theorem}
\noindent
The proof of Theorem \ref{thm:1} makes use of the following
\begin{lemma}
\label{lem:1}
Let $m\in \Nset$ and assume that $1,\rho_0,\rho_1,\ldots,\rho_m$ are $\Qset$-independent,
$(z_n)$ is a convergent sequence in $\Cset$, and at least one of the $2m$
numbers $c_1,\ldots,c_{2m}\in \Cset$ is non-zero. Then, for every $\alpha \in \Rset$,
the sequence
\begin{equation}\label{eq:xnlem1}
\bigl( n\rho_0 + \alpha \log n + \log |\xi_n|\bigr)
\end{equation}
is u.d.\ {\rm mod} $1$, where
$$
\xi_n := c_1e^{2\pi i n \rho_1} + c_2 e^{-2\pi i n \rho_1}+\ldots+
c_{2d-1}e^{2\pi i n \rho_m}+c_{2d}e^{-2\pi i n \rho_m}+z_n.
$$
\end{lemma}
\begin{proof}
Follows directly as in the proof of \cite[Lemma 2.9]{Berger05b}
which considers $\log|\mathfrak{Re}\xi_n|$ in \eqref{eq:xnlem1}.
\end{proof}
\begin{proof}[Proof of Theorem \ref{thm:1}]
By Lemma \ref{lem:6}(i), $\lim_{n \to \infty} P^n = P^*$ exists for the Markov
chain defined by $P$. Fix $(i,j) \in \{1,\ldots,d\}^2$. As the analysis of
$(P^{n+1} - P^n)^{(i,j)}$ is completely analogous, only $(P^n -P^*)^{(i,j)}$
will be considered here. If $(P^n -P^*)^{(i,j)}$ as given by \eqref{eq:7} is
not equal to zero for all but finitely many $n$, let $s_{i,j} \in \{1,\ldots,s\}$
be the minimal index such that $C_{s_{i,j}}^{(i,j)} \neq 0$. As in
\cite[p.224]{Berger05b}, to analyze \eqref{eq:7}, distinguish two cases.
\par\bigskip\noindent
\underline{Case 1:} $|\lambda_{s_{i,j}}| > |\lambda_{s_{i,j}+1}|$.
\par
\noindent
In this case $\lambda_{s_{i,j}}$ is a \emph{dominant}
eigenvalue, and it is real since otherwise
its conjugate would be an eigenvalue with the same modulus.
Equation \eqref{eq:7} can be written as
\begin{align*}
(P^n -  P^*)^{(i,j)}  & =  \sum\nolimits_{\ell=s_{i,j}}^d\lambda_\ell^n C_\ell^{(i,j)}
= |\lambda_{s_{i,j}}|^n\, n^{k_{s_{i,j}}^{(i,j)}} \sum\nolimits_{\ell=s_{i,j}}^d
\left(\frac{\lambda_{\ell}}{|\lambda_{s_{i,j}}|} \right)^n
\frac{C_\ell^{(i,j)}}{n^{k_{s_{i,j}}^{(i,j)}}}
\\
& =  |\lambda_{s_{i,j}}|^n\, n^{k_{s_{i,j}}^{(i,j)}} \left(
c_{s_{i,j}}^{(i,j)}
\left(\frac{\lambda_{s_{i,j}}}{|\lambda_{s_{i,j}}|} \right)^n +
\zeta_{i,j}(n)\right) \, ,
\end{align*}
where
\[ c_{s_{i,j}}^{(i,j)} := \lim\nolimits_{n\to\infty}
n^{-k_{s_{i,j}}^{(i,j)}} C_{s_{i,j}}^{(i,j)}\neq 0\, ,
\]
and $\zeta_{i,j}(n)\to 0$ as $n \to \infty$ because $\lambda_{s_{i,j}}$
is a dominating eigenvalue. Therefore,
$$
\log \big|(P^n - P^*)^{(i,j)}\big| = n \log |\lambda_{s_{i,j}}| +
k_{s_{i,j}}^{(i,j)} \log n + \log |c_{s_{i,j}}^{(i,j)}|
+\eta_n \, ,
$$
with $\eta_n = \log \left| 1 + \zeta_{i,j}(n) e^{-i n \arg \lambda_{s_{i,j}}}/ c_{s_{i,j}}^{(i,j)}\right|$.
Since $\eta_n \to 0$ and $\log |\lambda_{s_{i,j}}|$ is irrational, the sequence
$(P^n - P^*)^{(i,j)}$ is Benford by Lemma \ref{lem:3} and the fact that $(x_n + \alpha \log n + \beta)$
is u.d.\ mod $1$ whenever $(x_n)$ is (e.g.\ \cite[Lem.\ 2.8]{Berger05b}).
\par\bigskip\noindent
\underline{Case 2:} $|\lambda_{s_{i,j}}| = |\lambda_{s_{i,j}+1}| = \ldots =
|\lambda_{t_{i,j}}| =: |\lambda_{i,j}|$ for some $t_{i,j} >
s_{i,j}$.
\par
\noindent
Here several \emph{different} eigenvalues of the same magnitude
occur, such as e.g.\ conjugate
pairs of non-real eigenvalues. Let $k^{(i,j)}$ be the maximal
degree of the polynomials $C_{\ell}^{(i,j)}$, $\ell = s_{i,j}, \ldots, t_{i,j}$.
As in Case $1$, express \eqref{eq:7} as
\[
(P^n - P^*)^{(i,j)} = |\lambda_{i,j}|^n  n^{k^{(i,j)}} \left(
c_{s_{i,j}}^{(i,j)} \left(
\frac{\lambda_{s_{i,j}}}{|\lambda_{s_{i,j}}|} \right)^n + \ldots +
c_{t_{i,j}}^{(i,j)} \left(
\frac{\lambda_{t_{i,j}}}{|\lambda_{t_{i,j}}|} \right)^n +
\zeta_{i,j}(n)\right)\, ,
\]
where $c_{\ell}^{(i,j)} := \lim_{n \rightarrow \infty}
n^{-k^{(i,j)}} C_{\ell}^{(i,j)} \in \Cset$ for
$\ell=s_{i,j},\ldots,t_{i,j}$, with $c_{\ell}^{(i,j)} \neq 0$ for at
least one $\ell$, and $\zeta_{i,j}(n) \to 0$ as $n \to
\infty$. Consequently,
\begin{align*}
\log  \big|(P^n - P^*)^{(i,j)}\big|  = n & \log
|\lambda_{i,j}| + k^{(i,j)} \log n \\
& + \log \left| c_{s_{i,j}}^{(i,j)} \left(
\frac{\lambda_{s_{i,j}}}{|\lambda_{s_{i,j}}|}\right)^n + \ldots +
c_{t_{i,j}}^{(i,j)} \left(
\frac{\lambda_{t_{i,j}}}{|\lambda_{t_{i,j}}|}\right)^n +
\zeta_{i,j}(n) \right|\, .
\end{align*}
Write
$\lambda_{\ell}$ as $\lambda_{\ell} = |\lambda_{i,j}|e^{i
\arg \lambda_{\ell}}$ for $\ell = s_{i,j}, \ldots, t_{i,j}$, and hence
\begin{equation*}%
\begin{split}
\log \big|(P^n - P^*)^{(i,j)}\big|  =
n & \log |\lambda_{i,j}| + k^{(i,j)} \log n \\
& + \log \left| c_{s_{i,j}}^{(i,j)} e^{i n \arg \lambda_{s_{i,j}}} +
\ldots + c_{t_{i,j}}^{(i,j)} e^{i n \arg \lambda_{t_{i,j}}} +
\zeta_{i,j}(n) \right|\, .
\end{split}
\end{equation*}
Since $P$ is nonresonant, Lemma \ref{lem:1} applies with $m=t_{i,j} - s_{i,j}+1$ and
$\rho_0 = \log|\lambda_{i,j}|$, $\rho_1 = \frac1{2\pi} \arg \lambda_{s_{i,j}}$, \dots,
$\rho_m = \frac1{2\pi} \arg \lambda_{t_{i,j}}$. Thus $(P^n - P^*)^{(i,j)}$ is Benford.
\end{proof}

\begin{exam1} \label{ex:7}
(The general two-dimensional case)\\
Let $d=2$ and $ P = \begin{bmatrix}
       1-x & x \\
       y   & 1-y
\end{bmatrix}
$ with $x,y \in (0,1)$. By Feller \cite[p. $432$]{Feller50},
\begin{equation}\label{eq:fel}
P^n = \frac{1}{x+y}
\left[ \begin{array}{cc}
               y & x \\
               y & x
\end{array} \right] + \frac{(1-x-y)^n}{x+y}
\left[ \begin{array}{rr}
              x & -x \\
             -y & y
\end{array}\right] \, ,
\end{equation}
from which it is clear that $\lambda_1 =1$, $\lambda_2 = 1-x-y$, and
$
P^* = {\displaystyle \frac{1}{x+y}}
\left[ \begin{array}{cc}
              y & x \\
              y & x
\end{array} \right]
$. It follows from (\ref{eq:fel}) that each component of $(P^n - P^*)$ and
$(P^{n+1}-P^{n})$ is a multiple of $(\lambda_2^n)$. By Theorem~\ref{thm:1},
the Markov chain with transition probability matrix $P$ is Benford whenever
$\log |1-x-y|$ is irrational. On the other hand, by Lemma \ref{lem:5} $P$ is
not Benford if $\log |1-x-y|\in \Qset$. Thus for $d=2$, nonresonance is (not
only sufficient but also) necessary for $P$ to be Benford. For $d\ge 3$, this
is no longer true, see Example \ref{ex:15} below.
\end{exam1}

\begin{exam1}\label{ex:8}
(The general three-dimensional case)\\
Let $d=3$ and $ P= \left[ \begin{array}{ccc}
           x_1  &  x_2 & 1-x_1-x_2 \\
           y_1  &  y_2 & 1-y_1-y_2 \\
           z_1  &  z_2 & 1-z_1-z_2
           \end{array}\right]
$,
where $x_1,x_2,y_1,y_2,z_1,z_2 \in (0,1)$ are such that $x_1+x_2, y_1+y_2, z_1 +z_2$
all lie between $0$ and $1$. Solving the characteristic equation yields the
eigenvalues $\lambda_1 = 1$ and $\lambda_{2,3} = a\pm\sqrt{a^2-b}$, with
$$
a ={\textstyle \frac{1}{2}}(x_1+y_2-z_1-z_2) \quad \mbox{and} \quad
b = x_1y_2 - x_1z_2 + y_1z_2 - x_2y_1  + x_2z_1 - y_2z_1\, .
$$
Furthermore, using
$$
c = 1-y_2+z_1-y_2z_1 + x_2(-y_1+z_1)+x_1(-1+y_2-z_2)+z_2+y_1z_2 \ne 0\, ,
$$
one finds that
\[
P^* = \frac{1}{c}\left[ \begin{array}{c c c}
    z_1-y_2z_1+y_1z_2 &
    x_2 z_1 + z_2 - x_1 z_2 &
    1 - x_1 - x_2 y_1 - y_2 + x_1 y_2 \\
    z_1-y_2z_1+y_1z_2 &
    x_2 z_1 + z_2 - x_1 z_2 &
    1 - x_1 - x_2 y_1 - y_2 + x_1 y_2 \\
    z_1-y_2z_1+y_1z_2 &
    x_2 z_1 + z_2 - x_1 z_2 &
    1 - x_1 - x_2 y_1 - y_2 + x_1 y_2 \\
    \end{array}\right] \, .
\]
If $a^2\ne b$, then $ P^n-P^* = \lambda_2^n B_2 + \lambda_3^n B_3$,
where $B_{\ell}$ for $\ell=2,3$ are as in (\ref{eq:specdecom}).
There are two cases to consider:
\begin{enumerate}
\item \underline{$a^2>b$.}\\
In this case, $\lambda_{2,3}$ are real, and the dominant eigenvalue must be
identified. If $a >0$, then $|\lambda_2| > |\lambda_3|$, hence $\lambda_2$ is
dominant. If $B_2^{(i,j)} \neq 0$ for all $(i,j)\in \{1,2,3\}^2$, then the
Markov chain defined by $P$ is Benford if $\log |\lambda_2|$ is irrational.
In case there also exists $(i,j)$ with $B_2^{(i,j)} = 0$ yet $B_3^{(i,j)}\ne 0$,
then for $P$ to be Benford $\log |\lambda_3|$ has to be irrational as well.
For $a<0$ the roles of $\lambda_2$ and $\lambda_3$ have to be interchanged.
If $a=0$, then $P$ is resonant but may still be Benford, see Example \ref{ex:15}(ii).

\item \underline{$a^2<b$.}\\
Here $\lambda_{2,3}$ are conjugate and non-real, with
$|\lambda_2|=|\lambda_3|=\sqrt{b}$. Thus $P$ is nonresonant if and only if
the numbers $1, \frac12 \log b, \frac1{2\pi} \arctan \sqrt{b/a^2 - 1}$ are
$\Qset$-independent.
\end{enumerate}

\noindent
Finally, if $a^2 = b$ then $\lambda_2 = \lambda_3 = a$, so $P$ is Benford
whenever $\log|a|$ is irrational.
\end{exam1}
\par
\noindent
The next example shows that for a Markov chain to be Benford, nonresonance
is not necessary in general.
\begin{exam1}\label{ex:15}
(Markov chains that are resonant yet Benford)
\begin{enumerate}[(i)]
\item
Eigenvalues with rational argument: Let $d=3$ and
$
P = \left[ \begin{array}{ccc}
0.4 & 0.5 & 0.1\\
0.4 & 0.3 & 0.3 \\
0.6 & 0.1 & 0.3
\end{array} \right]
$. The eigenvalues are $\lambda_1 = 1$ and $\lambda_{2,3} = \pm 0.2i$.
With $\Lambda_0=\{0.2i\}$ therefore $\frac1{2\pi} \arg \Lambda_0 =\{\frac14\} \subset \Qset$,
so $P$ is resonant. However, spectral decomposition shows that
$B_3=\overline{B_2}$, i.e., $B_2, B_3$ are conjugates, and each component
of $B_2$ has non-zero real {\em and\/} imaginary part. Thus for every
$(i,j)\in \{1,2,3\}^2$,
\[
\big|(P^n-P^*)^{(i,j)}\big| =
\big|2\mathfrak{Re} (0.2i)^n B_2^{(i,j)}\big| = \left\{
\begin{array}{lll}
2\cdot 0.2^n \big|\mathfrak{Re} B_2^{(i,j)}\big| & & \mbox{if } n \mbox{ is even},\\[2mm]
2\cdot 0.2^n \big|\mathfrak{Im} B_2^{(i,j)}\big| & & \mbox{if } n \mbox{ is odd},\\
\end{array}
\right.
\]
and $(P^n - P^*)^{(i,j)}$ is Benford.

\item
Two real eigenvalues of opposite sign: Let $d=3$ and
$P = \left[ \begin{array}{ccc}
0.4 & 0.5 & 0.1\\
0.7 & 0.2 & 0.1 \\
0.4 & 0.2 & 0.4
\end{array}\right]
$. The eigenvalues are $\lambda_1 = 1$ and $\lambda_{2,3} = \pm 0.3$.
It can be checked that each component of $B_2 \pm B_3$ is non-zero.
Thus for every $(i,j)\in \{1,2,3\}^2$,
\[
(P^n-P^*)^{(i,j)}= 0.3^n\left(B_2^{(i,j)}+(-1)^n B_3^{(i,j)}\right),
\]
which is Benford because $\log 0.3\not\in\Qset$.
\end{enumerate}
\end{exam1}
\vspace{3 mm}

\noindent \textbf{Remarks on general Markov chains:}

\medskip

(i) Theorem \ref{thm:1} can not be applied to Markov chains
that fail to be irreducible. However, every finite-state Markov
chain can be decomposed into classes of recurrent and transient
states. Hence, the transition matrix $P$ can be block-partitioned
as
\[
P = \left[ \begin{array}{ccccc}
           P_1     &  0           &  \cdots &  0        &    0      \\
           0       &  P_2         &  0      &  0        &    0      \\
       \vdots      &              &  \ddots &           &    \vdots \\
           0       &  0           &         &  P_r      &    0      \\
           B^{(1)} & B^{(2)}      &  \cdots & B^{(r)}   &    A      \\
           \end{array}\right] ,
\]
where $P_1, P_2,\ldots, P_r$ are the transition matrices of the $r$
disjoint recurrent classes, and $B^{(1)}, B^{(2)},\ldots, B^{(r)}$ denote
the transition probabilities from the collection of transient states
into each recurrent class. As $n\to \infty$,
\[
P^n   =  \!  \left[ \begin{array}{ccccc}
           P_1^n        &  0        & \cdots  & 0        &    0      \\
           0            &  P_2^n    & 0       & 0        &    0      \\
          \vdots        &           & \ddots  &          &    \vdots \\
       0            & 0         &         & P_r^n    &    0      \\
           L_n^{(1)}    & L_n^{(2)} & \cdots  & L_n^{(r)}&    A^n    \\
           \end{array}\right]  \to
           \left[ \begin{array}{ccccc}
           P_1^*     &  0     & \cdots  & 0     &    0      \\
           0         &  P_2^* & 0       & 0     &    0      \\
           \vdots    &        & \ddots  &       &    \vdots \\
       0         & 0      &         & P_r^* &    0      \\
       SB^{(1)}P_1^*     & SB^{(2)}P_2^*  & \cdots  & SB^{(r)}P_r^* &    0 \\
           \end{array}\right] ,
\]
where $L_n^{(j)} = \sum_{\ell =0}^{n-1} A^{\ell} B^{(j)}P_j^{n - \ell - 1}$
for $j=1,2,\ldots, r$, and $S=\sum_{k=0}^{\infty} A^k$. Theorem \ref{thm:1}
can be applied separately to the transition matrices $P_j$ associated with
the recurrent classes. Consequently, if $P_1,P_2,\ldots, P_r$ are
Benford, then the corresponding components of $P$ are also Benford.
Additionally, if $A$ is nonresonant, then that part follows BL
as well. The only remaining parts are formed by the sequences
$\bigl(L_n^{(j)}\bigr)$ and depend on the (nonautonomous) summation of the powers of $A$.
Their Benford properties are beyond the scope of this paper.

\medskip

(ii) For an irreducible Markov chain that is not aperiodic, but rather
periodic with period $p>1$, Definition \ref{def:2} still makes sense, provided
that $P^*$ is understood as the unique row-stochastic matrix with $P^* P = P^*$.
However, such a chain cannot be Benford since for every $(i,j)\in \{1,\ldots , d\}^2$
one can choose $k\in \{0, \ldots , p-1\}$ such that
$$
|(P^n - P^*)^{(i,j)}| = (P^*)^{(i,j)} > 0 \, , \quad
\forall n \in \Nset \backslash (k + p\Nset) \, .
$$
Similarly, each component of $(P^{n+1} - P^n)$ equals zero at least
$(p-2)/p$ of the time and thus cannot be Benford either whenever
$p\ge 3$. The distribution of significands of $(P^{n+1} - P^n)^{(i,j)}$
observed in this situation is a convex combination of BL and a pure point
mass, see \cite[Cor.\ 6]{Berger08}. Only in the case $p=2$ is it possible
for each component of $(P^{n+1} - P^n)$ to be either Benford or eventually
zero.

\medskip

(iii) Although this paper deals with finite-state Markov chains only,
it is worth noting that chains with {\em infinitely\/} many states may also obey
BL in one way or the other. For a very simple example, let $0<\rho<1$ and
consider the homogeneous random walk on $\Zset$ with
$$
P^{(i,j)} = \left\{
\begin{array}{lll}
\rho^2 & & \mbox{if } j=i-1 \, , \\
2\rho(1-\rho) & & \mbox{if } j=i \, , \\
(1-\rho)^2 & & \mbox{if } j=i+1 \, , \\
0 & & \mbox{otherwise}\, .
\end{array}
\right.
$$
Clearly, this Markov chain is irreducible and aperiodic. It is (null-)recurrent
if $\rho=\frac12$, and transient otherwise. For all $(i,j)\in \Zset^2$ and $n\in \Nset$,
$$
(P^n)^{(i,j)} = \left(
\begin{array}{c}
2n \\ n + i - j
\end{array}
\right) \rho^{n+i - j} (1-\rho)^{n-i + j} \, ,
$$
and an application of Stirling's formula shows that $(P^n)^{(i,j)}$
is Benford if and only if $\log \bigl( 4 \rho (1-\rho)\bigr)$ is
irrational. For all but countably many $\rho$, therefore,
$(P^n)^{(i,j)}$ is Benford for every $(i,j)$. Note that one of the
excluded values is $\rho=\frac12$, i.e.\ the recurrent case. For {\em
recurrent\/} chains virtually every imaginable behavior of
significant digits or significands can be manufactured by means of
advanced ergodic theory tools, see \cite{Berger10} and the
references therein.


\section{Almost all Markov chains are Benford}\label{s:thm2}

The second main theoretical objective of this paper is to show that
Benford behavior is typical in finite-state Markov chains. Indeed,
if the transition probabilities of the chain are chosen at random,
independently and in any continuous manner, then the chain
almost always, i.e.\ with probability one, obeys BL. To formulate
this more precisely, the following terminology will be used.

\begin{defin}
\label{def:5} A \emph{random} ({\em $d$-state})
\emph{Markov chain} is a random $d \times d$-matrix
$\bm{P}$, defined on some probability space $(\Omega, \mathcal{F},
\mathbb{P})$ and taking values in $\mathcal{P}_d$, i.e., each row
$\bm{X}_1,\ldots,\bm{X}_d$ of $\bm{P}$ is a random vector taking
values in the standard $d$-simplex
\[
\Delta_d := \left\{(x_1,\ldots,x_d) \in \mathbb{R}^d : x_j \geq 0 \mbox{
 for all  } 1 \leq j \leq d, \mbox{ and } \sum\nolimits_{j=1}^d x_j = 1
 \right\}.
\]
A random vector $\bm{X}: \Omega \to \Delta_d$ is {\em continuous\/} if its
distribution on $\Delta_d$ is continuous w.r.t.\ the (normalised) Lebesgue
measure on $\Delta_d$, that is, if $\mathbb{P}(\bm{X}\in A)=0$ whenever
$A\subset \Delta_d$ is a nullset.
\end{defin}

\noindent
With this terminology, it is the purpose of the present section to illustrate
and prove

\begin{theorem}
\label{thm:2} If the transition probabilities (i.e.\ the rows) of a random
Markov chain $\bm{P}$ are independent and continuous, then $\bm{P}$ is
Benford with probability one.
\end{theorem}

Before giving a full proof for Theorem \ref{thm:2}, the special case of
a random two-state chain will be examined to show how independence and
continuity together allow the application of Theorem \ref{thm:1}. The
case $d=2$ is especially transparent since the eigenvalue functions are
simple and explicit, unlike for the general case where the eigenvalues
are only known implicitly, and the Implicit Function Theorem has to be
resorted to.

\begin{exam1}\label{ex:9}
Consider the random two-state Markov chain
\[
\bm{P}= \begin{bmatrix}
           1-\bm{X}  &  \bm{X} \\
           \bm{Y}    &  1-\bm{Y}
           \end{bmatrix},
\]
where the random variables $\bm{X}$ and $\bm{Y}$ are i.i.d.\ (absolutely)
continuous random variables on the unit interval $[0,1]$.
Since $\bm{X}$ and $\bm{Y}$ are continuous, each of the four entries
of $\bm{P}$ is strictly positive with probability one, so the chain
is irreducible and aperiodic with probability one. Since $\bm{P}$ is
random, the second-largest eigenvalue is a random variable $\bm{Z}$
which, by Example~\ref{ex:7}, satisfies $\bm{Z} = 1-\bm{X}-\bm{Y}$.
Since $\bm{X}$ and $\bm{Y}$ are independent and continuous, $\bm{Z}$
is also continuous, and hence the probability that
$\bm{Z}$ is in any given countable set is zero. But this implies
that the probability of $\log |\bm{Z}|$ being rational is zero, which
in turn shows that with probability one, $\bm{P}$ is nonresonant, and
hence Benford, by Theorem \ref{thm:1}.
\end{exam1}

Similarly to the analysis of Newton's method in \cite{Berger07},
a key property in the present Markov chain setting is the
\emph{real-analyticity} of certain functions, notably the
eigenvalue functions. Recall that a function $f: U \to \Cset$
is {\em real-analytic\/} whenever it can, in the neighborhood
of every point in its domain $U$ (an open subset of
$\Rset^\ell$ for some $\ell \ge 1$), be written as a convergent
power series. Clearly, every real-analytic function is $C^{\infty}$,
i.e.\ has derivatives of all orders. An important property of
real-analytic functions not shared by arbitrary $\Cset$-valued
$C^{\infty}$-functions defined on $U$ is that the set $\{x\in U :f(x) =0\}$
is a nullset unless $f$ vanishes identically on $U$.

The proof of Theorem \ref{thm:2} will be based on several preliminary
results. First, given $a=(a_1,\ldots,a_d) \in \mathbb{C}^d$, let
$p_a : \mathbb{C} \rightarrow \mathbb{C}$ denote the polynomial
\[
p_a(z) = z^d + a_1 z^{d-1} +\ldots+ a_{d-1}z + a_d\,  .
\]
By the Fundamental Theorem of Algebra, $p_a$ has exactly $d$ zeroes
(counted with multiplicities). If $p_a$ and $p_a'$, or more
generally, if $p_a$ and $p_b$ with $a \neq b$ have a common zero
then a universal polynomial relation must necessarily be satisfied
by $a$ and $b$. Only a special case of this elementary fact is
required here, and since no reference is known to the authors, a
proof is included for completeness.

\begin{lemma}
\label{prop:1} For every integer $d > 1$, there exists a non-trivial
polynomial $Q_d$ in $2d-1$ variables with the following property:
Whenever $a=(a_1,\ldots,a_d) \in \mathbb{C}^d$, $b=(b_1,\ldots,b_{d-1})
\in \mathbb{C}^{d-1}$, and $p_a(z_0) = p_b(z_0)=0$ for some $z_0 \in
\mathbb{C}$, then $Q_d(a,b) := Q_d(a_1,\ldots,a_d,b_1,\ldots,b_{d-1})=0$.
\end{lemma}

\begin{proof}
For $d=2$, let $Q_2(a,b): = a_1b_1 - a_2 - b_1^2$ for all
$a=(a_1,a_2) \in
\mathbb{C}^2$ and  $b=b_1 \in \mathbb{C}$.
To see that $Q_2$ has the desired property, note that if $p_a(z_0) =
0=p_b(z_0)$, then $z_0^2 + a_1z_0 + a_2 = 0$ and $z_0 = -b_1$, hence
$Q_2(a,b)=0$. Assume now that $Q_d$ has already been constructed. For
every $a\in \Cset^{d+1}$ and $b\in \Cset^d$ let
$\rho = a_2 - b_2 -(a_1-b_1)b_1 \in \mathbb{C}$, as well as
\[
c = \bigl( a_3 - b_3-(a_1-b_1)b_2 ,\ldots, a_d - b_d -(a_1-b_1)b_d,
a_{d+1} - (a_1 -b_1)b_d \bigr) \in \mathbb{C}^{d-1},
\]
and define
\[
Q_{d+1}(a,b) := \rho^{1+ {\rm deg}\, Q_d}Q_d\left(b,\frac{c}{\rho}\right)\, ,
\]
where ${\rm deg} \left( \sum_{j}
c_jx_1^{n_{1,j}}x_2^{n_{2,j}}\ldots x_{\ell}^{n_{\ell,j}}\right) := \max \left\{ n_{1,j}
+\ldots + n_{\ell,j} : c_j \neq 0\right\}$.
Clearly, $Q_{d+1}$ is a polynomial in $2d+1$ variables, and $Q_{d+1}
\neq 0$. If $p_a(z_0) = p_b(z_0) = 0$ for some $z_0 \in \mathbb{C}$, then
\begin{equation}
\label{eq:8}
\begin{split}
0 & =  p_a(z_0) - \bigl( z_0 + (a_1-b_1)\bigr)p_b(z_0) \\
  & =  \sum\nolimits_{j=1}^{d-1} \left(a_{j+1}-b_{j+1}-(a_1-b_1)b_i\right)z_0^{d-j} + a_{d+1}-(a_1-b_1)b_d \, .
\end{split}
\end{equation}
If $\rho = 0$, then clearly $Q_{d+1}(a,b) =0$. Otherwise, it is easy
to check that \eqref{eq:8} implies $p_{c/\rho}(z_0)=0$, in which case
$Q_d(b,c/\rho)=0$, by assumption. In either case, therefore, $Q_{d+1}(a,b)=0$.
\end{proof}

\begin{cor}
\label{corr:1} For every integer $d > 1$, there exists a non-trivial
polynomial $Q_d^*$ in $d$ variables such that $Q_d^*(a) = 0$
whenever $p_a(z_0) = p_a'(z_0) = 0$ for some $z_0 \in \mathbb{C}$.
\end{cor}

\begin{proof}
Take $Q_d^* = Q_d(a,b)$ with $b = \left( \frac{d-1}{d}a_1,
\frac{d-2}{d}a_2, \ldots , \frac{2}{d}a_{d-2},
\frac{1}{d}a_{d-1}\right)$.
\end{proof}

This corollary will now be used to show
that if a stochastic matrix $P_0$ is invertible and has distinct
non-zero eigenvalues, then all stochastic matrices $P$ sufficiently
close to $P_0$ also are invertible and have distinct non-zero
eigenvalues. In fact, these eigenvalues are real-analytic
functions of $P$. To formulate this efficiently,
for every $P_0 \in \mathcal{P}_d$ and $\varepsilon > 0$ denote by
$B_{\varepsilon}(P_0)$ the open ball with radius $\varepsilon$
centered at $P_0$, i.e. $B_{\varepsilon}(P_0) = \bigl\{ P \in
\mathcal{P}_d : |P^{(i,j)}-P_0^{(i,j)}| < \varepsilon \mbox{ for all
} 1 \le i,j \le d \bigr\}$.

\begin{lemma}
\label{lem:7} Suppose $P_0 \in \mathcal{P}_d$ is invertible and has $d$
distinct non-zero eigenvalues. Then there exists $\varepsilon >0$
and and $d-1$ non-constant real-analytic functions $\lambda_2,\ldots,\lambda_d :
B_{\varepsilon}(P_0) \rightarrow \mathbb{C}$ such that, for every $P
\in B_{\varepsilon}(P_0)$,
\begin{enumerate}
\item $1,\lambda_2(P),\ldots,\lambda_d(P)$ are the eigenvalues of $P$,
and $\lambda_2(P) \cdot \ldots \cdot \lambda_d(P) \neq 0$;
\item $\lambda_i(P) \neq \overline{\lambda_j(P)}$ whenever $i \neq j$,
unless $\lambda_i = \overline{\lambda_j}$ on $B_{\varepsilon}(P_0)$.
\end{enumerate}
\end{lemma}

\begin{proof}
Note first that by the continuity of $(P,z) \mapsto \det (z I_{d\times d} - P)=\psi_P(z)$,
there exists $\delta >0$ such that every $P \in B_{\delta}(P_0)$ is invertible and
has distinct non-zero eigenvalues. Thus the characteristic polynomial $\psi_P$ of $P$
has $d-1$ distinct non-zero roots different from $1$. Let $z_0$ be one of those roots.
Since $z_0$ is a simple root, $\psi_{P_0}'(z_0) \neq 0$, so by the Implicit Function
Theorem \cite[Theorem 2.3.5]{Krantz02}, $z_0$ depends real-analytically on the
coefficients of $\psi_{P}$ which themselves are real-analytic (in fact polynomial)
functions of the entries of $P$. More formally, there exists $\varepsilon \le \delta$
and a real-analytic function $g : B_{\varepsilon}(P_0) \rightarrow \mathbb{C}$
with $g(P_0)=z_0$ such that $\psi_{P}\bigl(g(P)\bigr)=0$ for all $P \in
B_{\varepsilon}(P_0)$. Overall, there exists $\varepsilon > 0$ and $d-1$
real-analytic functions $\lambda_i : B_{\varepsilon}(P_0) \rightarrow \mathbb{C}$
satisfying (i); note that $\lambda_1 \equiv 1$ by Lemma \ref{lem:6}. To see that
$\lambda_2 , \ldots , \lambda_d$ are not constant on $B_{\varepsilon}(P_0)$,
suppose by way of contradiction that  $\lambda_i(P) = \lambda_i(P_0) \neq 1$
for some $2 \leq i \leq d$ and all $P \in B_{\varepsilon}(P_0)$. In this case,
the real-analytic function $P \mapsto \psi_P\bigl(\lambda_i(P_0)\bigr)$ vanishes
identically on $B_{\varepsilon}(P_0)$, and hence on all of $\mathcal{P}_d$. Since
$I_{d \times d} \in \mathcal{P}_d$, this obviously contradicts
$\psi_{I_{d \times d}} \bigl(\lambda_i(P_0)\bigr) = \left( \lambda_i(P_0) - 1
\right)^d \neq 0$. Consequently, none of the functions $\lambda_2
,\ldots, \lambda_d : B_{\varepsilon}(P_0) \rightarrow \mathbb{C}$ is
constant.

To show (ii), assume that $\lambda_i(P_1) = \overline{\lambda_j(P_1)}$
for some $i \neq j$ and $P_1 \in B_{\varepsilon}(P_0)$. Thus
$\lambda_i(P_1) \in \mathbb{C}\setminus \mathbb{R}$, since if
$\lambda_i(P_1)$ were real, then $\lambda_i(P_1) = \lambda_j(P_1)$,
which is impossible since the eigenvalues are distinct. Since all
matrices in $\mathcal{P}_d$ are \emph{real}, their non-real eigenvalues
occur in conjugate pairs. Hence, for all $P$ sufficiently close to
$P_1$, the number $\overline{\lambda_j (P)}$ is an eigenvalue of $P$
which, by continuity, can only be $\lambda_i(P)$. Consequently,
$\lambda_i$ and $\overline{\lambda_j}$ coincide locally near $P_1$
and therefore, by real-analyticity, on all of $B_{\varepsilon}(P_0)$.
\end{proof}

By means of the above auxiliary results, several almost sure properties
of random Markov chains can be identified.

\begin{lemma}
\label{prop:2} If the rows of the random Markov chain $\bm{P}$
are independent and continuous then, with probability one,
\begin{enumerate}
\item $\bm{P}$ is irreducible, aperiodic, and invertible;
\item $\bm{P}$ has $d$ distinct non-zero eigenvalues; and
\item $\bm{P}$ is nonresonant.
\end{enumerate}
\end{lemma}

\begin{proof}
Fix $\bm{P}$ and assume its rows
$\bm{X}_1,\ldots,\bm{X}_d$ are independent and continuous.

\medskip
\noindent
(i) Since each $\bm{X}_i$ is continuous, $\mathbb{P}(\bm{X}_i \in A)=0$
for every Lebesgue nullset $A \subset \Delta_d$, so in particular $\mathbb{P}(\bm{X}_{i,j}
\in \{0,1\}) =0$ for all $i$ and $j$. With probability one, therefore,
$\bm{P}^{(i,j)} \in (0,1)$ for all $i$ and $j$, and $\bm{P}$ is irreducible
and aperiodic. To see that $\bm{P}$ is almost surely invertible, note that
$P\mapsto \det P$ is a non-constant, real-analytic function on $\mathcal{P}_d$.
With $N= \bigl\{(x_1, \ldots, x_d)\in \Delta_d \times \ldots \times \Delta_d : \det(x_1, \ldots , x_d)=0
\bigr\}$,
\begin{align*}
\mathbb{P}(\det \bm{P}=0) & = \int\nolimits_{N} {\rm d}\mathbb{P} (x_1, \ldots , x_d) =
\int \!\! \cdots \!\! \int\nolimits_N {\rm d}\mathbb{P}(x_1) \ldots {\rm d} \mathbb{P}(x_d) \\
& = \int \!\! \cdots \!\! \int \left( \int\nolimits_N {\rm d}\mathbb{P} (x_1) \right) {\rm d}\mathbb{P}(x_2) \ldots
{\rm d}\mathbb{P}(x_d) = 0 \, ,
\end{align*}
where the second equality follows from the independence of $\bm{X}_1, \ldots, \bm{X}_d$,
the third from Fubini's theorem, and the fourth from the continuity of the $\bm{X}_i$.

\medskip
\noindent
(ii) There exist $d$ non-constant polynomial functions
$q_1,\ldots,q_d : \mathcal{P}_d \rightarrow \mathbb{R}$ such that
\[
\psi_P(z) = \det\left( zI_{d \times d} - P \right) = z^d +
q_1(P)z^{d-1} + \ldots + q_{d-1}(P)z + q_d(P)
\]
holds for all $P \in \mathcal{P}_d$ and $z \in \mathbb{C}$; for
example, $q_1(P) = - \sum_{i=1}^d P^{(i,i)}$ and $q_d(P)= (-1)^d\det P$.
Consequently, $q(P) := Q_d^*\bigl(q_1(P),\ldots,q_d(P)\bigr)$ defines a
non-constant real-analytic (in fact, polynomial) map $q : \mathcal{P}_d
\rightarrow \mathbb{R}$, and since $z_0$ is a multiple eigenvalue of
$P$ if and only if $\psi_P(z_0) = \psi_P'(z_0) = 0$, Corollary
\ref{corr:1} implies that
\[
\bigl\{ P \in \mathcal{P}_d : P \mbox{ has multiple eigenvalues}\,
\bigr\} \subset \bigl\{ P \in \mathcal{P}_d : q(P) = 0 \bigr\} \, .
\]
As before, by Fubini's Theorem $\mathbb{P}(q(\bm{P})=0)=0$,
showing that with probability one all eigenvalues of $\bm{P}$ are simple.

\medskip
\noindent
(iii) For every $\rho \in \mathbb{Q}$ define the real-analytic auxiliary
function $\Phi_{\rho} : \mathbb{R}^2 \rightarrow \mathbb{R}$ by
$\Phi_{\rho}(x) := (x_1^2 +x_2^2-10^{2\rho})^2$, and also
$\Theta: \Rset^4 \to \Rset$ as $\Theta(x) := \left(x_1^2+x_2^2-x_3^2-x_4^2 \right)^2$.
By (i) and (ii), $\bm{P}$ almost surely satisfies the hypotheses of Lemma
\ref{lem:7}, so let $P_0$, $\varepsilon$, and $\lambda_2,\ldots,\lambda_d$
be as in Lemma \ref{lem:7}, and define real-analytic functions $\Phi_{\rho,i}$
and $\Theta_{i,j}$ on $B_{\varepsilon}(P_0)$ as
\[
\Phi_{\rho,i}(P) := \Phi_{\rho} \bigl( \mathfrak{Re}\lambda_i(P),
\mathfrak{Im}\lambda_i(P)\bigr) = \left( |\lambda_i(P)|^2 -
10^{2\rho}\right)^2\, ,  \quad \forall i : 2 \le i \le d\, ,
\]
and, for all $2 \le i,j \le d$,
\[
\Theta_{i,j}(P) := \Theta\bigl( \mathfrak{Re}\lambda_i(P),
\mathfrak{Im}\lambda_i(P), \mathfrak{Re} \lambda_j(P), \mathfrak{Im}
\lambda_j(P) \bigr) = \left( |\lambda_i(P)|^2 -
|\lambda_j(P)|^2\right)^2.
\]
Finally, let $F_{\rho} : B_{\varepsilon}(P_0) \rightarrow
\mathbb{R}$ be defined as
\[
F_{\rho}(P) := \prod\nolimits_{i=2}^d \Phi_{\rho,i}(P) \cdot \prod\nolimits_{2\le i<j : \lambda_i
\neq \overline{\lambda_j}} \Theta_{i,j}(P) \, .
\]
The definition of $F_{\rho}$ becomes transparent upon noticing that $F_{\rho}(P)=0$
for some $\rho\in \Qset$ whenever $P$ is invertible and resonant.
Next, it will be shown that $F_{\rho}$ does not vanish identically on $B_{\varepsilon}(P_0)$.
To see this, note first that if $P \in B_{\varepsilon}(P_0)$, then also
$(1-\delta)P+\delta I_{d \times d} \in B_{\varepsilon}(P_0)$ for all
sufficiently small $\delta > 0$. Moreover, if $\Phi_{\rho,i}(P)=0$ for
some $i=2,\ldots,d$, then
\begin{align*}
\Phi_{\rho,i}\bigl((1-\delta)P+\delta I_{d \times d} \bigr) & =
\left(  \bigl((1-\delta)\mathfrak{Re} \lambda_i(P)+\delta\bigr)^2 +
(1-\delta)^2\mathfrak{Im} \lambda_i(P)^2-10^{2\rho}\right)^2 \\
& = \delta^2 \Bigl( (2-\delta) \left(\mathfrak{Re}
\lambda_i(P)-|\lambda_i(P)|^2\right) + \delta \bigl(( 1-\mathfrak{Re}
\lambda_i(P)\bigr) \Bigr)^2 > 0\, ,
\end{align*}
provided that $\delta >0$ is small enough. (Recall that $1-\mathfrak{Re} \lambda_i(P) > 0$
whenever $P \in B_{\varepsilon}(P_0)$.) Similarly, if $\Theta_{i,j}(P)=0$ for some
$2 \leq i < j \leq d$ with $\lambda_i \neq \overline{\lambda_j}$ and $\lambda_i(P) \neq 0$,
then a short calculation confirms that, for all $\delta>0$ sufficiently small,
\[
\Theta_{i,j}\bigl( (1-\delta)P + \delta I_{d \times d} \bigr) =
\delta^2(1-\delta)^2 \frac{|\lambda_i(P) - \lambda_j(P)|^2
|\lambda_i(P) - \overline{\lambda_j(P)}|^2}{|\lambda_i(P)|^2} > 0\, .
\]
Overall, $F_{\rho}$ does not vanish identically on $B_{\varepsilon}(P_0)$.
As every $P \in B_{\varepsilon}(P_0)$ is invertible,
$$
\bigl\{ P \in B_{\varepsilon}(P_0): P \mbox{ is resonant}\,
\bigr\} \subset \bigcup\nolimits_{\rho \in \mathbb{Q}} \bigl\{ P \in
B_{\varepsilon}(P_0) : F_{\rho} (P)= 0 \bigr\} \, .
$$
Since $F_{\rho}$ is real-analytic and non-constant, $\bigl\{
 P \in B_{\varepsilon}(P_0) : F_{\rho}(P) = 0 \bigr\}$ is a nullset
for every $\rho \in \Qset$, and so is
$\bigcup_{\rho \in \mathbb{Q}} \bigl\{  P \in B_{\varepsilon}(P_0) : F_{\rho} (P)= 0 \bigr\}$.
Analogously to (i) and (ii), therefore,
$\mathbb{P} \left( \bm{P} \mbox{ is resonant}\, \right) =
 0$.
\end{proof}

\begin{proof}[Proof of Theorem \ref{thm:2}]
Let $\bm{X}_1,\ldots,\bm{X}_d$ denote the random transition
probabilities (row vectors) of the random $d \times d$-matrix
$\bm{P}$. If $\bm{X}_1,\ldots,\bm{X}_d$ are independent and
continuous, then by Lemma \ref{prop:2}, $\bm{P}$ is almost surely
irreducible, aperiodic, and nonresonant. By Theorem \ref{thm:1},
this implies that $\bm{P}$ is Benford with probability one.
\end{proof}

\begin{rem}

\medskip

(i) It is clear that without independence, or without continuity,
Lemma \ref{prop:2} and Theorem \ref{thm:2} are generally false. For
example, for the conclusion of Lemma \ref{prop:2} to hold it is not
enough to assume that the distribution on $\Delta_d$ of each row of
$\bm{P}$ is atomless. As very simple examples show, under this
weaker assumption, $\bm{P}$ may, with positive probability, be
reducible and have multiple or zero eigenvalues. Even if Lemma
\ref{prop:2} (i,ii) hold with probability one, $\bm{P}$ may still
be resonant and not Benford. To see this, consider the random three-state
Markov chain
$$
\bm{P} = \frac1{40} \left[
\begin{array}{lll}
\bm{X} + 4 & \bm{X} & 36 - 2 \bm{X} \\
\bm{Y}     & \bm{Y} + 4 & 36 - 2 \bm{Y} \\
\bm{Z} + 2 & \bm{Z} + 2 & 36 - 2 \bm{Z}
\end{array}
\right] \, ,
$$
where $\bm{X}, \bm{Y}, \bm{Z}$ are independent and uniformly distributed on $[0,1]$.
The eigenvalues of $\bm{P}$ are
$$
\lambda_1 = 1 \, , \quad
\lambda_2 = 0.1 \, , \quad
\lambda_3 = {\textstyle \frac1{40}} (\bm{X} + \bm{Y} - 2 \bm{Z})\, .
$$
Note that $|\lambda_3| \le 0.05 < \lambda_2$. Clearly, $\bm{P}$ is resonant
with probability one, and Lemma \ref{prop:2}(iii) fails. Perhaps even more
importantly, Theorem \ref{thm:2} fails as well since, as spectral decomposition
shows, $B_2 \ne 0$ with probability one and hence $\mathbb{P}(\bm{P} \mbox{ is Benford}\,)=0$.

\medskip

(ii) With hardly any effort, the tools employed in the proof of Lemmas
\ref{lem:7} and \ref{prop:2} also yield a topological analogue of Theorem
\ref{thm:2}: Within the compact metric space $\mathcal{P}_d$, the matrices
that are irreducible, aperiodic, invertible and nonresonant form a
{\em residual\/} set, that is, a set whose complement is the countable
union of nowhere dense sets. Being Benford, therefore, is a typical
property for $P\in \mathcal{P}_d$ not only under a probabilistic perspective
but under a topological perspective as well.

\end{rem}


\section{Simulations}\label{s:sim}

In this section, numerical simulations will illustrate the theoretical
results of previous sections, and based on these simulations the
\emph{rate of convergence\/} towards BL will be discussed. Since it
is not possible to observe the empirical frequencies of infinite
sequences, $(P^n - P^*)$ and $(P^{n+1} - P^{n})$ are simulated up to
a predefined value of $n$, such as $n=1000$ or $n=10000$, and the
empirical distributions of first significant digits of each
component are compared to the Benford probabilities. For some
Markov chains, simulations up to $n=1000$ yield empirical frequencies
very close to BL, whereas for others even $n=10000$ does not give a
good approximation, although theoretically all chains considered here
follow BL. Thus, convergence rates towards BL may differ significantly.

\begin{exam}
\label{ex:11}
From Table \ref{tab:5}, it is clear that the sequences $(2^n)$,
$(n!)$, $(F_n)$ give different empirical frequencies for the
simulation up to $n=1000$. Compared to the other two, $(F_n)$ gives
empirical frequencies much closer to BL.
\end{exam}

Similarly, rates of convergence can be discussed for Markov chains.
The important question is what property is creating the difference
in convergence rates. Theorem \ref{thm:2} shows that every homogeneous
Markov chain chosen independently and continuously is Benford with
probability one. Besides irreducibility and aperiodicity, nonresonance
is crucial. Irreducibility and aperiodicity do not determine the rate
of convergence. This leaves nonresonance as the only source for different
rates of convergence. According to Definition \ref{def:4}, nonresonance is
based on the rational independence of $1$, $\log L_0$ and the elements
of $\frac{1}{2 \pi} \arg \Lambda_0$, provided that $\Lambda_0 \ne \emptyset$.
Thus, it is natural to expect this rational independence to be reflected
in some quantitative manner in the rate of convergence towards BL.

It is well known that there are infinitely many rational approximations
for a given accuracy to any irrational number. Let $x$ be an irrational
number. Given any $\varepsilon>0$, there exist infinitely many pairs
$(p,q) \in \Zset\times \Nset$ with $\gcd{(p,q)}=1$ and
$$
\label{eq:ratM} \left| x - \frac{p}{q}\right| < \varepsilon\, .
$$
One way to obtain rational approximations of irrational numbers is provided
by the method of continued fractions. Every irrational real number $x$
is represented uniquely by its continued fraction expansion
\[
x = a_0 + \frac{1}{a_1+{\displaystyle \frac{1}{a_2+{\displaystyle \frac{1}{a_3 + \cdots}}}}}
\, ,
\]
also denoted as $x=[a_0;a_1,a_2,a_3,\ldots]$, where $a_0 \in \Zset$ and
$a_n \in \Nset$ for $n \ge 1$ are referred to as the \emph{partial quotients\/}
of $x$. By \cite[Theorem 149]{Hardy60}, if $p_n$ and $q_n$ are
defined iteratively as
\begin{align*}
p_0 & = a_0\, , \quad   p_1 =  a_1a_0+1 \, ,   & p_n  = a_n p_{n-1}+p_{n-2} \, , \quad \forall n \ge 2\, ,\\
q_0 & = 1 \,  , \quad \:\: q_1 =  a_1   \, ,   & q_n  = a_n q_{n-1}+q_{n-2} \, , \quad \forall n \ge 2\, ,
\end{align*}
then, for all $n\in \Nset$,
\[
\frac{p_n}{q_n} = a_0 + \frac{1}{a_1+{\displaystyle \frac{1}{a_2+{\displaystyle \frac{1}{\dots + {\displaystyle \frac1{a_n}}}}}}}
=:[a_0; a_1,\ldots,a_n] \, ;
\]
the rational numbers $p_n/q_n$ are called the \emph{convergents}
of the continued fraction of $x$. Leaving aside trivial exceptions,
best rational approximations to an irrational $x$ are of the form
$p_n/q_n$, and
\begin{equation}
\label{eq:cfrac1} \left| x - \frac{p_n}{q_n} \right| <
\frac{1}{a_{n+1}q_n^2}\, ,\quad \forall n \ge 2.
\end{equation}
It is clear from (\ref{eq:cfrac1}) that $p_n/q_n$ yields a particularly
good approximation of $x$ when $a_{n+1}$ is large. Hence $x$ can be
rapidly approximated if its continued fraction expansion contains a
sequence of rapidly increasing partial quotients. On the other hand,
if $(a_n)$ does not grow fast (or at all), then it is difficult to
approximate $x$ by a rational number with small error, see \cite{Hardy60,Kuipers74}
for details. For example, \cite[Ch. 2, Theorem 3.4]{Kuipers74} asserts
that if $(a_n)$ is bounded for some $x$ then the distribution mod 1 of
$(nx)$ approaches the uniform distribution rather quickly. Thus
irrationals which are hard to approximate by rational numbers, due
to a small upper bound on, or slow growth of $(a_n)$, are also
the ones for which one expects to see fast convergence to Benford
probabilities. Specifically, for the golden ratio
$\frac{1+\sqrt{5}}{2} = [1;1,1,1,\ldots ]$, every $a_n$ has the smallest
possible value. Since $\big|\log F_n - n \log\frac{1+\sqrt{5}}{2}\big|\to 0$
as $n\to \infty$, this may explain why the convergence to
BL is faster for the Fibonacci sequence than for the other two sequences
in Example \ref{ex:11}. (See \cite{Schatte90} for further insights
on BL for continued fractions.)

It is important to note that $(a_n)$ is unbounded for almost every $x$,
\cite[Theorem 196]{Hardy60}. Hence, in most simulations it is not possible
to observe convergence as fast as for the Fibonacci sequence. However, to
highlight the difference in rates of convergence and irrationality, two
examples are studied. The first $50$ partial quotients are given for
every relevant irrational number that arises.

\begin{exam1}
\label{ex:12} (Markov chain showing {\em fast\/} convergence)\\
Let $d=3$ and $ P = \left[ \begin{array}{ccc}
                      0.25 & 0.35 & 0.40 \\
                      0.30 & 0.45 & 0.25 \\
                      0.65 & 0.15 & 0.20
                      \end{array} \right]
$.
The eigenvalues of $P$ are $\lambda_1 = 1$ and $\lambda_{2,3} =
-\frac1{20} \mp \frac{1}{20}\sqrt{21}$ , hence
$ \sigma(P)^+ \setminus \{ \lambda_1 \} =
\bigl\{-\frac1{20} - \frac{1}{20}\sqrt{21}, -\frac1{20} + \frac{1}{20}\sqrt{21} \bigr\}$.
Since $\log |\lambda_2|$ and $\log |\lambda_3|$ are
irrational and different, $P$ is nonresonant. Thus Theorem~\ref{thm:1} implies
that the Markov chain defined by $P$ is Benford.

Table \ref{tab:1} shows the empirical frequencies of significant
digits for the first $1000$ and $10000$ terms of $(P^{n} - P^*)$,
respectively; the behavior of $(P^{n+1} - P^n)$ is very similar.

\begin{table}[ht]
\footnotesize{
\begin{center}
\begin{tabular}{c c c c c c c c c | c}
  \cline{1-10}
  & & & & & & & & & \\[-2mm]
  $(1,1)$ & $(2,1)$ & $(3,1)$ & $(1,2)$ &
  $(2,2)$ & $(3,2)$ & $(1,3)$ &
  $(2,3)$ & $(3,3)$ & Benford  \\[2.5mm]
  \hline \hline
  & & & & & & & & & \\[-3mm]
    0.300 &   0.301 &   0.300 &   0.303 &   0.303 &   0.299 &   0.300 &   0.306 &   0.300 &   0.30103 \\
    0.175 &   0.177 &   0.177 &   0.176 &   0.174 &   0.176 &   0.178 &   0.174 &   0.175 &   0.17609 \\
    0.126 &   0.124 &   0.123 &   0.125 &   0.125 &   0.125 &   0.124 &   0.124 &   0.127 &   0.12493 \\
    0.098 &   0.096 &   0.100 &   0.096 &   0.096 &   0.097 &   0.096 &   0.098 &   0.097 &   0.09691 \\
    0.078 &   0.081 &   0.079 &   0.080 &   0.080 &   0.079 &   0.079 &   0.078 &   0.077 &   0.07918 \\
    0.068 &   0.067 &   0.065 &   0.068 &   0.067 &   0.066 &   0.068 &   0.067 &   0.069 &   0.06694 \\
    0.058 &   0.059 &   0.059 &   0.056 &   0.057 &   0.060 &   0.059 &   0.057 &   0.058 &   0.05799 \\
    0.050 &   0.050 &   0.051 &   0.051 &   0.052 &   0.052 &   0.050 &   0.050 &   0.052 &   0.05115 \\
    0.047 &   0.045 &   0.046 &   0.045 &   0.046 &   0.046 &   0.046 &   0.046 &   0.045 &   0.04575 \\[2mm]
  \hline
  & & & & & & & & & \\[-3mm]
    0.3008 &   0.3009 &   0.3009 &   0.3011 &   0.3012 &   0.3008 &   0.3011 &   0.3017 &   0.3010 &   0.30103 \\
    0.1761 &   0.1762 &   0.1764 &   0.1762 &   0.1758 &   0.1762 &   0.1763 &   0.1759 &   0.1760 &   0.17609 \\
    0.1249 &   0.1250 &   0.1247 &   0.1248 &   0.1251 &   0.1249 &   0.1249 &   0.1249 &   0.1250 &   0.12493 \\
    0.0971 &   0.0968 &   0.0972 &   0.0969 &   0.0968 &   0.0970 &   0.0968 &   0.0969 &   0.0970 &   0.09691 \\
    0.0792 &   0.0793 &   0.0791 &   0.0792 &   0.0793 &   0.0790 &   0.0790 &   0.0790 &   0.0789 &   0.07918 \\
    0.0668 &   0.0669 &   0.0666 &   0.0670 &   0.0670 &   0.0668 &   0.0672 &   0.0671 &   0.0673 &   0.06694 \\
    0.0582 &   0.0582 &   0.0582 &   0.0580 &   0.0578 &   0.0582 &   0.0580 &   0.0577 &   0.0579 &   0.05799 \\
    0.0510 &   0.0509 &   0.0512 &   0.0510 &   0.0512 &   0.0514 &   0.0510 &   0.0512 &   0.0513 &   0.05115 \\
    0.0459 &   0.0458 &   0.0457 &   0.0458 &   0.0458 &   0.0457 &   0.0457 &   0.0456 &   0.0456 &   0.04575 \\[2mm]
  \hline \hline\\[-8mm]
\end{tabular}
\end{center}
}
\caption{\textit{Comparing empirical frequencies for the first significant digits
with Benford probabilities for the first $1000$ (top half) and $10000$ (bottom half)
terms of the sequences $(P^{n} - P^*)^{(i,j)}$, where $P$ is the transition
probability matrix in Example \ref{ex:12}.}} \label{tab:1}
\end{table}

\noindent
Since $|\lambda_2|>|\lambda_3|$, all that matters is how well
\begin{align*}
\log |\lambda _ 2| = [ -1; & 2, 4, 8,  1, 5, 1, 6, 3, 1, 2, 2,  1,  1, 2, 1, 1,  2, 1, 66, 5, 1, 1, 2, 1, 3, \\
&   1, 2, 1,  1, 3, 1, 3, 2, 3, 2, 7,  3, 86, 1, 1, 1,  1, 1, 26, 3, 1, 5, 3, 1, 5, \ldots ]
\end{align*}
is approximated by rational numbers. From the above, $a_n \le 86$ for all
$1 \le n \le 50$, and a rapid increase of quotients is not observed. This continued
fraction expansion should be compared to the ones in the example below.
\end{exam1}

\begin{exam1}
\label{ex:13} (Markov chain showing {\em slow\/} convergence)\\
Let $d=3$ and $ P= \left[ \begin{array}{ccc}
                    0.8 &  0.1 & 0.1 \\
                    0.3 &  0.3 & 0.4 \\
                    0.4 &  0.0 & 0.6
                     \end{array}\right],
$ with eigenvalues $\lambda_1 = 1$ and $\lambda_{2,3} =
\frac{7}{20}\pm \frac{1}{20}\sqrt{3}~i$. Thus
$\sigma(P)^+ \setminus \bigl\{ \lambda_1 \} = \{
\frac{7}{20}+\frac{1}{20}\sqrt{3}~i \bigr\} = :\Lambda_0$,
and the behavior of significant digits is governed by the two
irrational numbers
\begin{align*}
\log  |\lambda_2| =  [ -1; & 1, 1, 3, 1, 7, 1, 15, 1, 2, 1, 1, 7, 1, 6, 2, 1, 3, 1, 1, 2, 4, 1, 1, 2, 3,\\
&   8, 1, 2, 1, 1, 2, 1,  2, 1, 7, 1, 1, 2, 1,33, 1, 2, 1, 2, 1, 1, 11, 1, 24,8, \ldots ] \, , \\[2mm]
{\textstyle \frac1{2\pi}}\arg\lambda_2 = [ 0; & 25, 1, 9, 3,168,2,  1, 1,32, 1, 6, 3, 1, 9, 1, 1, 92,2, 13,2, 1, 1, 10,2, 5, \\
  & 1, 3, 1, 1, 1, 1,  3, 1, 2, 7, 1, 5, 1, 1, 4, 1, 3,14, 3,10, 1, 1, 3, 1, 3, \ldots ]\, .
\end{align*}
Note that $\max_{n=1}^{50}a_n = 33$ for $\log  |\lambda_2|$, whereas
$\max_{n=1}^{50}a_n = 168$ for $\frac1{2\pi}\arg \lambda_2$. When
compared with Example \ref{ex:12}, the repeated early high values in
the continued fraction expansion of $\frac1{2\pi}\arg \lambda_2$
suggest a somewhat slower convergence to BL. As shown in Table
\ref{tab:3}, this slower convergence is clearly recognizable in
simulations of $(P^n - P^*)$; again the behavior of $(P^{n+1} - P^n)$
is very similar.

\begin{table}[hbt]
\vspace*{4mm}
\small{
\begin{tabular}{c c c c c c c c c | c}
  \cline{1-10}
   & & & & & & & & & \\[-2mm]
  $(1,1)$ & $(2,1)$ & $(3,1)$ & $(1,2)$ &
  $(2,2)$ & $(3,2)$ & $(1,3)$ &
  $(2,3)$ & $(3,3)$ & Benford \\[2.5mm]
  \hline \hline
    & & & & & & & & & \\[-3mm]
    0.302 &   0.313 &   0.311 &   0.327 &   0.290 &   0.286 &  0.293  & 0.298 &   0.297 &   0.30103 \\
    0.176 &   0.169 &   0.170 &   0.152 &   0.178 &   0.181 &  0.192  & 0.181 &   0.184 &   0.17609 \\
    0.127 &   0.137 &   0.136 &   0.137 &   0.110 &   0.114 &  0.103  & 0.122 &   0.122 &   0.12493 \\
    0.096 &   0.081 &   0.085 &   0.087 &   0.101 &   0.101 &  0.123  & 0.105 &   0.102 &   0.09691 \\
    0.075 &   0.079 &   0.080 &   0.086 &   0.093 &   0.091 &  0.061  & 0.071 &   0.074 &   0.07918 \\
    0.074 &   0.080 &   0.084 &   0.072 &   0.055 &   0.056 &  0.061  & 0.063 &   0.069 &   0.06694 \\
    0.072 &   0.049 &   0.048 &   0.046 &   0.055 &   0.054 &  0.061  & 0.083 &   0.074 &   0.05799 \\
    0.039 &   0.047 &   0.043 &   0.046 &   0.056 &   0.055 &  0.070  & 0.038 &   0.041 &   0.05115 \\
    0.039 &   0.045 &   0.043 &   0.047 &   0.062 &   0.062 &  0.036  & 0.039 &   0.037 &   0.04575 \\[2mm]
  \hline
   & & & & & & & & & \\[-3mm]
  0.2998 & 0.3150 & 0.3158 & 0.3167 & 0.2910 & 0.2922 & 0.2938 & 0.2982 & 0.2981 & 0.30103 \\
  0.1798 & 0.1620 & 0.1610 & 0.1570 & 0.1865 & 0.1867 & 0.1877 & 0.1816 & 0.1821 & 0.17609 \\
  0.1312 & 0.1397 & 0.1399 & 0.1354 & 0.1069 & 0.1079 & 0.1090 & 0.1232 & 0.1236 & 0.12493 \\
  0.0943 & 0.0828 & 0.0837 & 0.0859 & 0.1002 & 0.0983 & 0.1192 & 0.1033 & 0.1027 & 0.09691 \\
  0.0716 & 0.0825 & 0.0825 & 0.0965 & 0.0877 & 0.0887 & 0.0640 & 0.0702 & 0.0698 & 0.07918 \\
  0.0753 & 0.0789 & 0.0782 & 0.0610 & 0.0570 & 0.0561 & 0.0600 & 0.0682 & 0.0694 & 0.06694 \\
  0.0665 & 0.0476 & 0.0478 & 0.0496 & 0.0550 & 0.0546 & 0.0618 & 0.0748 & 0.0741 & 0.05799 \\
  0.0416 & 0.0458 & 0.0462 & 0.0478 & 0.0575 & 0.0570 & 0.0680 & 0.0412 & 0.0409 & 0.05115 \\
  0.0399 & 0.0457 & 0.0449 & 0.0501 & 0.0582 & 0.0585 & 0.0365 & 0.0393 & 0.0393 & 0.04575 \\[2mm]
  \hline \hline\\[-8mm]
\end{tabular}
}
\caption{\textit{Comparing empirical frequencies for the first significant digits
with Benford probabilities for the first $1000$ (top half) and $10000$ (bottom half)
terms of the sequences $(P^{n} - P^*)^{(i,j)}$, where $P$ is the transition
probability matrix in Example \ref{ex:13}.}} \label{tab:3}
\end{table}
\end{exam1}


\section{Applications}\label{s:app}

In scientific calculations using digital computers and floating
point arithmetic, roundoff errors are inevitable, and as Knuth
points out in his classic text \emph{The Art of Computer
Programming} \cite[pp.\ 253--255]{Knuth97},

\begin{quote}
In order to analyze the average behavior of floating-point
arithmetic algorithms (and in particular to determine their average
running time), we need some statistical information that allows us
to determine how often various cases arise \dots  [If, for example, the]
leading digits tend to be small [that] makes the most obvious
techniques of average error estimation for floating-point
calculations invalid. The relative error due to rounding is
usually \dots  \,\! more than expected.
\end{quote}

Thus for the problem of numerical estimation of $P^*$ from $P^n$, it
is important to study the distribution of significant digits (or,
equivalently, the fraction parts of floating-point numbers) of the
components of $(P^{n} - P^*)$ and $(P^{n+1} - P^{n})$.

Theorem~\ref{thm:2} above shows that the components of both $(P^{n}
- P^*)$ and $(P^{n+1} - P^{n})$ typically exhibit exactly the type
of nonuniformity of significant digits alluded to by Knuth: Not only
do the first few significant digits of the differences between the
components of the successive $n$-step transition matrices $P^n$ and
the limiting distribution $P^*$, as well as the differences between
$P^{n+1}$ and $P^{n}$ tend to be small but, much more specifically,
they typically follow BL.

This prevalence of BL has important practical implications for
estimating $P^*$ from $P^n$ using floating-point arithmetic. One
type of error in scientific calculations is overflow (or underflow),
which occurs when the running calculations exceed the largest (or
smallest, in absolute value) floating-point number allowed by the
computer. Feldstein and Turner show that \cite[p.\ 241]{Feldstein86},
``[u]nder the assumption of the logarithmic distribution of
numbers (i.e., BL) floating-point addition and subtraction can
result in overflow and underflow with alarming frequency \dots ''.
Together with Theorem~\ref{thm:2}, this suggests that special
attention should be given to overflow and underflow errors in any
computer algorithm used to estimate $P^*$ from $P^n$.

Another important type of error in scientific computing is due
to roundoff. In estimating $P^*$ from $P^n$, for example, every
stopping rule, such as ``stop when n=1000'' or ``stop when the
components in $(P^{n+1} - P^{n})$ are less than $10^{-10}$'', will
result in some error, and Theorem~\ref{thm:2} shows that
this difference is generally Benford. In fact, justified by
heuristics and by the extensive empirical evidence of BL
in other numerical calculations, analysis of roundoff errors has
often been carried out under the \emph{hypothesis} of a logarithmic
statistical distribution (cf. \cite[p.\ 326]{Feldstein86}).
Therefore, as Knuth pointed out, a naive assumption of uniformly
distributed significands in the calculations tends to
underestimate the average relative roundoff error in cases where the
actual statistical distribution of fraction parts is skewed toward
smaller leading significant digits, as is the case in BL. To obtain
a rough idea of the magnitude of this underestimate when the true
statistical distribution is BL, let $\bm{X}$ denote the absolute roundoff
error at the time of stopping the algorithm, and let $\bm{Y}$ denote the
fraction part of the approximation at the time of stopping. Then the
relative error is $\bm{X}/\bm{Y}$, and assuming that $\bm{X}$ and $\bm{Y}$
are independent random variables, the average (i.e., expected)
relative error is simply $\mathbb{E}\bm{X} \cdot \mathbb{E}(1/\bm{Y})$. Thus
if $\bm{Y}$ is assumed to be uniformly distributed on $[1, 10)$, ignoring
the fact that $\bm{Y}$ is Benford creates an average underestimation of
the relative error by {\em more than one third\/} (cf.\ \cite{Berger07}).

As one potential application of Theorems~\ref{thm:1}
and~\ref{thm:2}, it should be possible to adapt the current
plethora of BL-based goodness-of-fit statistical tests for detecting
fraud (e.g.\ \cite{Carslaw88}), to the problem of detecting whether
or not a sequence of realizations of a finite-state process
originates from a Markov chain, i.e., whether or not the process is
Markov. By Theorem~\ref{thm:2}, conformance with BL for the differences
$P^{n+1} - P^{n}$ is typical in finite-state Markov chains, so a
standard (e.g.\ chi-squared) goodness-of-fit to BL of the empirical
estimates of the differences between $P^{n+1}$ and $P^{n}$ may help
detect non-Markov behavior.
\bigskip

\noindent {\large{\textbf{Acknowledgements}}}
\smallskip
\par\noindent
T.P.\ Hill wishes to express his gratitude to the Vrije Universiteit
Amsterdam, and to the Econometrics Department, especially Professor
Henk Tijms, for their hospitality during several visits when part of
this research was conducted.


\renewcommand{\baselinestretch}{1.0}

\end{document}